\documentclass[a4paper]{amsart}
\usepackage[all]{xy}
\usepackage{amsmath}
\usepackage{amsfonts,amssymb}
\usepackage{hyperref}

\usepackage{booktabs}

\numberwithin{equation}{section}
\newtheorem{satz}{Satz}[section]
\newtheorem{proposition}[satz]{Proposition}
\newtheorem{theorem}{Theorem}
\newtheorem*{theorem*}{Theorem}

\newtheorem{lemma}[satz]{Lemma}

\newtheorem{corollary}[satz]{Corollary}

\theoremstyle{definition}
\newtheorem*{remark*}{Remark}
\newtheorem{remark}[satz]{Remark}
\newtheorem*{notation}{Notation}

\newtheorem{example}[satz]{Example}
\newtheorem{assumption}{Assumption}

\newcommand{\tensor}{\otimes}

\newcommand{\map}[1]{\stackrel{#1}{\longrightarrow}}

\newcommand{\eqweil}[1]{\stackrel{#1}{=}}

\newcommand{\un}[1]{\ensuremath{\protect\underline{#1}}}

\def\PGL{\textsf{PGL}}

\DeclareMathOperator{\Pic}{Pic}

\DeclareMathOperator{\Bun}{Bun}

\DeclareMathOperator{\Chain}{Chain}

\DeclareMathOperator{\Higgs}{Higgs}

\DeclareMathOperator{\Hom}{Hom}
\DeclareMathOperator{\cHom}{\mathcal{H}om}
\DeclareMathOperator{\cEnd}{\mathcal{E}nd}

\DeclareMathOperator{\Spec}{Spec}
\DeclareMathOperator{\Ext}{Ext}

\DeclareMathOperator{\Aut}{Aut}

\DeclareMathOperator{\End}{End}

\DeclareMathOperator{\ad}{ad}

\DeclareMathOperator{\rank}{rank}

\DeclareMathOperator{\tr}{tr}

\DeclareMathOperator{\rk}{rk}

\DeclareMathOperator{\Var}{Var}
\DeclareMathOperator{\id}{id}
\DeclareMathOperator{\wt}{wt}

\DeclareMathOperator{\coker}{coker}

\DeclareMathOperator{\sst}{ss}
\DeclareMathOperator{\st}{s}
\DeclareMathOperator{\nst}{div}
\DeclareMathOperator{\opp}{opp}
\DeclareMathOperator{\sat}{sat}

\def\gr{\mathrm{gr}}

\def\K0hat{\widehat{K}_0(\Var)}
\def\univ{\textrm{\tiny univ}}

\def\1halb{\frac{1}{2}}

\def\pprime{{\prime\prime}}

\def\sxymat{\xymatrix@C=1.5ex@R=0.8ex}
\def\grp{$\xymatrix{ R\times_{X}R  \ar[r]^-{\mu} & R \ar@<1ex>[r]^-{s}\ar@<-1ex>[r]_-{t} & X}$}
\def\dar{\ar@<-0.5ex>[r]\ar@<0.5ex>[r]}
\def\tar{\ar[r]\ar@<1ex>[r]\ar@<-1ex>[r]}
\newcommand{\dmap}[2]{\ar@<-0.5ex>[r]_-{#2}\ar@<0.5ex>[r]^-{#1}}
\newcommand{\dotarrow}[2]{\xymatrix{{#1}\ar@{..>}[r]&{#2}}}
\def\cart{\ar@{}[dr]|{\square}}

\def\cA{\mathcal{A}}

\def\cE{\mathcal{E}}
\def\cF{\mathcal{F}}
\def\cG{\mathcal{G}}

\def\cL{\ensuremath{\mathcal{L}}}

\def\cO{\mathcal{O}}

\def\cQ{\mathcal{Q}}

\def\bA{{\mathbb A}}

\def\bC{{\mathbb C}}

\def\bG{{\mathbb G}}
\def\bH{{\mathbb H}}

\def\bN{{\mathbb N}}

\def\bP{{\mathbb P}}
\def\bQ{{\mathbb Q}}
\def\bR{{\mathbb R}}

\def\bZ{{\mathbb Z}}

\begin{document}
\SelectTips{cm}{}

\title[Intersection form vanishes]{The intersection form on moduli spaces of twisted $PGL_n$-Higgs bundles vanishes}
\author[J. Heinloth]{Jochen Heinloth}
\address{Universit\"at Duisburg--Essen, Fachbereich Mathematik, Universit\"atsstrasse 2, 45117 Essen, Germany}
\email{Jochen.Heinloth@uni-due.de}
\begin{abstract}Hausel and Rodriguez-Villegas conjectured that the intersection form on the moduli space of stable $\PGL_n$-Higgs bundles on a curve vanishes if the degree is coprime to $n$. In this note we prove this conjecture. Along the way we show that moduli spaces of stable chains are irreducible for stability parameters larger than the stability condition induced form stability of Higgs bundles.
\end{abstract}

\dedicatory{\today}
\maketitle

The aim of this article is to prove a conjecture of Hausel and Rodriguez-Villegas on the middle cohomology of the moduli space of stable Higgs bundles on a curve.  
The setup for this conjecture is as follows. 
Let $C/\bC$ be a smooth projective curve of genus $g>1$ over the complex numbers and fix a pair of coprime integers $n,d$ with $n\geq 2$. 
A Higgs bundle on $C$ is a pair $(\cE,\theta)$, where $\cE$ is a vector bundle on $C$ and $\theta\colon \cE \to \cE\tensor \Omega$ a morphism of $\cO_C$-modules. We denote by $M_{n}^{d}$ the coarse moduli space of semistable Higgs bundles of rank $n$ and degree $d$ on $C$. Since we assumed $(n,d)=1$ semistability and stability agree and in this case it is known that  $M_n^d$ is a smooth quasi-projective variety. 

There are natural analogs of these objects for $\PGL_n$-bundles, parametrizing pairs $(\overline{\cE},\theta)$, where $\overline{\cE}$ is a principal $\PGL_n$-bundle and $\theta \in H^0(C,\ad(\overline{\cE})\tensor\Omega)$.  It turns out that the coarse moduli space $M_{\PGL_n}^d$ of semistable $\PGL_n$-Higgs bundles is the quotient of $M_n^{d}$ by the action of $M_1^0\cong T^*\Pic_C^0$ (see \cite{HauselMirror}). This allows us to work with $M_n^d$ for all geometric considerations. Our main result is:
\begin{theorem}[Conjecture of Hausel--Rodriguez-Villegas {\cite[4.5.1]{HRV}}]\label{main}
The intersection form on $H^*(M_{\PGL_n}^d)$ is trivial. Equivalently the forgetful map $H^*_c(M_{\PGL_n}^d)\to H^*(M_{\PGL_n}^d)$ is $0$.
\end{theorem}
In \cite{HRV} this conjecture appeared as a consequence of a series of conjectures on the structure of the cohomology of $M_n^d$.
The case $n=2$ was shown by T. Hausel \cite{HauselIntersectionForm} much earlier and it was originally  motivated by a conjecture of Sen (see \cite{HauselSduality}).
By a different method this was reproven  in \cite[Theorem 1.1.7]{HRV}.  This case was later used in the proof of the P=W conjecture for $n=2$ in \cite{dCHM}. 

Our approach was motivated by the observation \cite[Proposition 4]{HVortrag} that the conjecture admits an equivalent formulation in terms of the Hitchin fibration, which is reminiscent of Ng\^o's support theorem \cite[Th\'eor\`eme 7.8.3]{Ngo}. 
To explain this we need to recall some properties of the Hitchin fibration.
Let $$h\colon M_{n}^{d} \to \cA=\oplus_{i=1}^n H^0(C,\Omega^{\tensor i})$$ denote the Hitchin map, which is defined as $h(\cE,\theta):= (-1)^{i+1}\tr( \wedge^i \theta)$. It induces a map $$h_{\PGL_n}\colon M_{\PGL_n}^d \to \cA_{\PGL_n} := \oplus_{i=2}^n H^0(C,\Omega^{\tensor i}).$$
Both of these maps are known to be projective and flat by a theorem of Nitsure \cite{Nitsure}.

Since $M_n^d$ is smooth and  the map $h\colon M_{n}^{d} \to \cA$ is projective the decomposition theorem implies that the complexes  $\bR h_*\bQ$ and $\bR h_{\PGL_n,*} \bQ$ decompose as a direct sum of shifted perverse sheaves on $\cA$ and $\cA_{\PGL_n}$. By \cite[Proposition 4]{HVortrag} Theorem \ref{main} is equivalent to the statement that the perverse sheaves occurring in the decomposition $\bR h_{\PGL_n,*} \bQ$ are isomorphic to the middle extension of their restriction to $\cA_{\PGL_n}-\{ 0 \}$. In particular the conjecture implies that the cohomology of the nilpotent cone $h^{-1}_{\PGL_n}(0)$ is determined by the topology of the restriction of the Hitchin fibration to $\cA_{\PGL_n}-\{0\}$. This is reminiscent of Ng\^o's support theorem, which he used to show that if one replaces $\Omega$ by $\Omega(D)$ for some positive divisor $D$, then the supports of the restriction of $\bR h_{\PGL_n,*} \bQ$ to a certain open subset $\cA^{ell}_{\PGL_n} \subset \cA_{\PGL_n}$ are all $\cA^{ell}_{\PGL_n}$ itself. This was extended to the locus of reduced spectral curves in the work of Chaudouard--Laumon \cite{CL1} and moreover  they showed recently that the support theorem does extend to all of $\cA_{\PGL_n}$, again if one replaces $\Omega$ by $\Omega(D)$ for some divisor $D$ with $\deg(D)>0$ \cite{ChaudouardLaumon}. However, for the original space of Higgs bundles the dimension estimates used in these proofs seem to allow for potential summands supported in $0$.

The strategy of our proof is the following. First one observes that for dimension reasons, the conjecture is only interesting for cohomology classes in the middle degree $H_c^{\dim M_{\PGL_n}^d}(M_{\PGL_n}^d)$. This cohomology group turns out to be generated by the cycles classes of the irreducible components of $h_{\PGL_n}^{-1}(0)$. Since these components are the quotients of the components of $h^{-1}(0)$ by the action of $\Pic_C$, we will try to find for any irreducible component $F$ of $h^{-1}(0)$ --- except for the one parametrizing stable bundles with trivial Higgs field --- a deformation $F_t$ that is contained in a fiber $h^{-1}(t)$ with $t\neq 0$. This will be done by deforming the Higgs field of a bundle fixed under the $\bG_m$--action, without changing the underlying bundle. This allows us to deform components in $h^{-1}(0)$ into fibers over $t\in \cA$ which correspond to reducible and often non-reduced spectral curves. 

To conclude from there we use a crucial remark by Tamas Hausel that the self-intersection of the component parametrizing stable bundles vanishes because its Euler characteristic is $0$.

Two technical problems arise in this naive approach. First, there is a modular description of the components of the nilpotent cone, but for our application we need to show that the irreducible components of these moduli spaces are indexed by their natural numerical invariants. Since this result may be of independent interest, let us give the precise statement. It is known that the fixed points of the $\bC^*$--action on $M_n^d$ can be described as moduli spaces of stable chains of vector bundles.

Recall that a chain of vector bundles is a collection $(\cE_r \map{\phi_r} \cE_{r-1} \map{\phi_{r-1}} \dots \map{\phi_1} \cE_0)$, where $\cE_i$ are vector bundles on $C$ and $\phi_i$ are arbitrary morphisms of $\cO_C$-modules. There is a natural notion of stability for chains depending on a parameter $\un{\alpha} \in \bR^{r+1}$ and for any $\un{\alpha}$ there exists a projective coarse moduli space of $\un{\alpha}$-semistable chains of rank $\un{n}$ and degree $\un{d}$ \cite{AGPS}. The stack of $\un{\alpha}$-semistable chains will be denoted by $\Chain_{\un{n}}^{\un{d},\un{\alpha}-\sst}$ (see section \ref{notations} for more details and the notion of critical stability parameters). 
In the second part of the article we will show:
\begin{theorem}\label{irreducible}
For $\un{n}\in \bN^{r+1}, \un{d}\in \bZ^{r+1}, \alpha\in \bR^{r+1}$ such that $\alpha_{i+1}-\alpha_i>2g-2$ for all $i$ and such that $\alpha$ is not a critical value the stack
$\Chain_{\un{n}}^{\un{d},\alpha-\sst}$ is irreducible.
\end{theorem}
In \cite{BGPG} this result was proven for $r=1$ and for $r=2$ many cases are shown in \cite{AGPS}. We follow the same strategy as these references, which apply a variation of the stability condition. To carry this out we need to study for which stability parameters the corresponding flip loci can have the same dimension as the  whole moduli space. Surprisingly, a detailed analysis of the result \cite[Proposition 4.5]{AGPS} can be used to show that this can happen exactly at those walls that were found in \cite{GPH} to give necessary conditions for the existence of stable chains. This is the only place where we use the assumption that our ground field $k=\bC$, because the proof of \cite[Proposition 4.5]{AGPS} relies on an analytic argument.

The second technical problem in our strategy is to control the closures of the irreducible components as well as their deformations under the $\bG_m$--action. To do this we study $\bG_m$--equivariant maps of $\bA^1$ and $\bP^1$  to the moduli stack of Higgs bundles. The basic argument which helps us to control these maps is the fact that the degree of equivariant line bundles on $\bP^1$ can be read of the weights of $\bG_m$ on  the stalks at the fixed points of the action. Choosing an ample, $\bG_m$--equivariant line bundle on $M_n^d$ this allows one to find a natural ordering of the fixed point components and this turns out to give sufficiently many restrictions on the closures of our components to conclude our argument.

The structure of the article is as follows. In the first section we introduce notations and give the argument for the vanishing of the intersection form Theorem \ref{main}, following the lines of the argument sketched above using Theorem \ref{irreducible} as an assumption. Section 2 then proves this assumption in case the ground field is $\bC$. A reader interested only in Theorem \ref{irreducible} could skip forward to Section 2 after reading through the notations introduced in Section 1.1.

\noindent{\bf Acknowledgments:} I am greatly indebted to T.\ Hausel. Discussions with him are the reason why this article exists and it was his idea to use the Poincar\'e-Hopf theorem to finish the proof of Theorem \ref{main}. A large part of this work was done while visiting his group at EPFL. A part of the work was funded through the SFB/TR 45 of the DFG.
\section{The intersection form}

\subsection{Fixed point strata, chains and the stratification of the global nilpotent cone}\label{notations}
In this section we introduce the notation and recall the basic results on the Hitchin fibration that we will use. We refer to \cite{GPHS} for more detailed references. We will denote by $\Bun_n^d$ the moduli stack of vector bundles of rank $n$ and degree $d$ on $C$ and by   $\Higgs_{n}^{d}$ the moduli stack of Higgs bundles of rank $n$ and degree $d$.

The multiplicative group $\bG_m$ acts on $\Higgs_{n}^{d}$, by $t\cdot(\cE,\theta):= (\cE,t\theta)$. This also defines an action on the coarse moduli space $M_{n}^{d}$. This action is equivariant with respect to the Hitchin map, if one defines the $\bG_m$--action on $H^0(C,\Omega^i)$ by multiplication by $t^i$.  Since this action on $\cA=\oplus_{i=1}^n H^0(C,\Omega^i)$ is contracting for $t\to 0$, the fixed points for the action are contained in $h^{-1}(0)$.

\begin{remark}\label{reduction}
Since the action contracts $M_n^d$ to $h^{-1}(0)$ one has $$H^*(M_n^d)\cong H^*(h^{-1}(0)).$$ In particular $H^*(M_n^d)=0$ for $*> 2 \dim h^{-1}(0) = \dim M_n^d$ and by Poincar\'e duality $H^*_c(M_n^d)=0$ for $*< \dim M_n^d$. The same argument holds for $M_{\PGL_n}^d$. Therefore in order to prove Theorem \ref{main} it is sufficient to prove it for the middle degree $*=\dim M_{\PGL_n}^d$. The top cohomology of the projective equidimensional varieties $h^{-1}(0),h_{\PGL_n}^{-1}(0)$ is freely generated by classes indexed by their irreducible components, so that $H^{\dim M_{\PGL_n}^d}_c(M_{\PGL_n}^d)$ is freely generated by the cycle classes of the irreducible components of $h_{\PGL_n}^{-1}(0)$. Since $h_{\PGL_n}^{-1}(0)=h^{-1}(0)/\Pic_C$ it will suffice to describe the components of $h^{-1}(0)$. 
\end{remark}
	
Hitchin \cite{Hitchin} and Simpson observed that the fixed points of the $\bG_m$--action can be described as moduli spaces of stable chains (e.g. \cite[Lemma 9.2]{HT1}). To clarify the different conventions used for chains and Higgs bundles we will denote by
$$F_{\un{n}}^{\un{d}} := \left\langle \big( (\cE_i)_{i=0\dots r},  (\theta_i\colon \cE_{i} \to \cE_{i-1} \tensor \Omega)_{i=1,\dots r}\big) \left| {\cE_i\in \Bun_{n_i}^{d_i} \atop (\oplus \cE_i, \oplus \theta_i) \in M_{n}^{d} }\right. \right\rangle$$
the fixed point stratum in $h^{-1}(0)$ and by
$$\Chain_{\un{n}}^{\un{d}} =\langle (\cE_r \map{\phi_r} \cE_{r-1} \to \dots \map{\phi_1} \cE_0 | \cE_i \in \Bun_{n_i}^{d_i} \rangle $$
the stack of chains of rank $\un{n}$ and degree $\un{d}$. 

We will prove in Section 2 (see Remark \ref{stab}) that the following assumption holds:
\begin{assumption}\label{assum1}
If $(n,d)=1$ then for any  $\un{n},\un{d}$ with $\sum n_i=n,\sum d_i=d$ the variety $F_{\un{n}}^{\un{d}}$ is irreducible. $(\star_n)$
\end{assumption}
\begin{remark}
The above assumption is well known for $n=2,3$ by the explicit description of the strata $F_{\un{n}}^{\un{d}}$  given by Hitchin and Thaddeus. For $n=4$ the assumption follows from the results of \cite{AGPS} and \cite{BGPG}.
\end{remark}

We will denote by $F_{\un{n}}^{\un{d},\pm}$ the locally closed subvarieties of $M_n^d$ given by the points such that the limits of the points under the $\bG_m$--action for $t\to 0$ (resp. $t\to \infty$) lie in $F_{\un{n}}^{\un{d}}$. The theorem of Bia\l ynicki-Birula \cite[Theorem 4.1]{BialBirula} implies that the maps $F^{\un{d},\pm}_{\un{n}}\to F_{\un{n}}^{\un{d}}$  given by mapping a point to its limit under the $\bG_m$--action are smooth, locally trivial fibrations such that the fibers are affine spaces. The dimension of the fibers of the projection from $F^{\un{d},+}_{\un{n}}$ equals the dimension of the part the tangent space at a point in $F_{\un{n}}^{\un{d}}$ on which $\bG_m$ acts with positive weights and similarly the dimension of the fibers of the projection from $F^{\un{d},-}_{\un{n}}$ is the dimension of the subspace of the tangent space on which $\bG_m$ acts with negative weights.

Laumon proved that the Hitchin fibration is Lagrangian \cite{LaumonNilp} which implies (this was observed in \cite[Section 9]{HT1}) that the map $F_{\un{n}}^{\un{d},+} \to F_{\un{n}}^{\un{d}}$ is a fibration with fibers of dimension $\dim \cA$. For the negative strata we know that $F_{\un{n}}^{\un{d},-}\subset h^{-1}(0)$. Moreover, since  $M_{n}^{d}$ is smooth of dimension $2\dim \cA$ and the fibers of $F_{\un{n}}^{\un{d},+} \to F_{\un{n}}^{\un{d}}$ are of dimension $\dim \cA$ we conclude that $\dim F_{\un{n}}^{\un{d},-} = \dim \cA$. Finally, since $h$ is proper \cite{Nitsure} we know that every point in $h^{-1}(0)$ has limit points under the $\bG_m$--action and therefore $h^{-1}(0) = \cup F_{\un{n}}^{\un{d},-}$. 

Thus, assumption $(\star_n)$ implies that the closures of the $F_{\un{n}}^{\un{d},-}$ are the irreducible components of $h^{-1}(0)$.

\subsection{A partial ordering of the irreducible components of $h^{-1}(0)$}

To study the intersection form we will need some information on which of the fixed point strata $F_{\un{m}}^{\un{e}}$ can intersect the closure of a component $F_{\un{n}}^{\un{d},-}$. Over the complex numbers the Morse function $\mu=|| \theta \theta^*||_{L^2}$ introduced by Hitchin would give restrictions on the numerical invariants. For us the following algebraic analogue will be useful: Let us denote by $\cL_{\det}$ the line bundle on $\Bun_n^d$ given by the inverse of the determinant of cohomology, i.e., its fiber at a bundle $\cE$ is given by $$\cL_{\det}|_{\cE}=\det H^1(C,\End(\cE)) \tensor (\det H^0(C,\End(\cE)))^{-1}.$$ The same formula also defines a line bundle on the stack of $\PGL_n$ bundles $\Bun_{\PGL_n}^d$, which we will denote by the same symbol.
 The pull back of $\cL_{\det}$ to $\Higgs_n^d$ defines a line bundle $\cL_{M_n^d}$ on $M_n^d$, because the central automorphisms of a bundle $\cE$ act trivially on the fibers of $\cL_{\det}$. It is known that the induced line bundle on $M_{\PGL_n}^d$ is (relatively) ample with respect to $h_{\PGL_n}$ \cite[Theorem 5.10 and its proof, Remark 5.12]{Nitsure}. 

Given a bundle $(\cE,\theta) \in h^{-1}(0)$, which is not a fixed point, the closure of the $\bG_m$ orbit of the bundle defines an equivariant  map $\bP^1 \to h^{-1}(0) \subset M_{n}^{d}$, which by definition induces a  2-commutative diagram:
$$\xymatrix{
\bP^1 \ar[d]\ar[r]^f &  M_{n}^{d} \ar[d]^{\text{forget}}\ar@{=>}[dl] \\
[\bP^1/\bG_m] \ar[r]^{\overline{f}} & \Bun_n^d.
}$$
Since $\cL_{M_{\PGL_n}^d}$ is ample on $h_{\PGL_n}^{-1}(0)$ and the composition $\bP^1\map{f} M_n^d \to M_{\PGL_n}^d$ is not the constant map, we know that $\deg(f^*(\cL_{M_n^d}))>0$. Moreover, the degree of an equivariant line bundle on $\bP^1$ turns out to be determined by the weights of the $\bG_m$ action on the fibers of this bundle over $0,\infty$. These weights are the numerical invariants we will use. In order to fix our sign conventions let us recall the relation between the degree and the weights in detail: 
\begin{remark}
Let $\cL$ be a $\bG_m$--equivariant line bundle on $\bP^1$. 
We fix coordinates of the standard affine charts of $\bP^1$: Around $0$ we choose $\bA^1_0:= \bA^1 = \Spec k[x]$ and $\bA^1_{\infty} := \bA^1 = \Spec k[y]$. 
The standard action of $\bG_m=\Spec k[t,t^{-1}]$ on $\bA^1_0$ is given on points by $(t,x)\mapsto tx$ and on coordinate rings by $x \mapsto t \tensor x$. On global functions 
$H^0(\bA^1,\cO_{\bA^1})$ this induces $t.f (x) =f(t^{-1}x)$, i.e., it is given by $x \mapsto t^{-1}x$.

Given a $\bG_m$--action on a line bundle $\cL$ on $\bA^1$ we can choose any global generator $e\in H^0(\bA^1,\cL) \cong k[x]$. Since the space of globally generating sections is of dimension 1 the action will then be given by $t.e := t^{w_0}.e$ for some integer $w_0$. In particular if $w_0\geq 0$ the invariant sections are spanned by $x^{w_0} e$ and there are no invariant sections if $w_0<0$.

Similarly, restricting a $\bG_m$--equivariant line bundle on $\bP^1$ to $\bA_\infty^1$ we can choose a trivialization $e_\infty \in H^0(\bA^1_\infty,\cL)$ and then $t.e_\infty = t^{w_\infty}e_\infty$ for some $w_\infty \in \bZ$.

We will denote the weight of the $\bG_m$--action on the stalks of $\cL$ at $0$ and $\infty$ by $$\wt(\bG_m,\cL_0):=w_0, \wt(\bG_m,\cL_\infty):= w_\infty.$$

The degree of the equivariant bundle $\cL$ is determined by the weights as follows: Choose trivializations $e_0,e_\infty$ over $\bA_0^1,\bA^1_\infty$ as before. As the sections are equivariant with respect to the $\bG_m$--action we must have $e_0|_{\bG_m } = y^d e_\infty|_{\bG_m}$ with $d=\deg(\cL)$. So we find $d+w_{\infty} = w_0$, i.e., 
$$\deg{\cL}= \wt(\bG_m,\cL_0) - \wt(\bG_m,\cL_\infty).$$
\end{remark}
\begin{example}\label{degwt}
In our applications equivariant bundles will be obtained from maps $$f\colon [\bA^1/\bG_m] \to \Bun_{n}^d.$$ For any filtered vector bundle $\tilde{\cE}_r \subset \dots \tilde{\cE}_0=\cE$ with subquotients $\cE_i=\tilde{\cE}_i/\tilde{\cE}_{i-1}$ 
we have a natural degeneration $f_{\cE_\bullet}\colon [\bA^1/\bG_m] \to \Bun_{n}^d$, given by the action of $\bG_m$ on the space of iterated extensions $\Ext(\cE_0,\dots,\cE_r)$ which is induced from the action of $\bG_m$ on $\oplus \cE_i$, where $\cE_i$ is of weight $i$. 

The fiber of $f^*_{\cE_\bullet}\cL_{\det}|_0$ is given by $\det H^1(C,\cEnd(\oplus \cE_i)) \tensor \det(H^0(C,\cEnd(\oplus\cE_i)))^{-1}$. By transport of structure we can compute the weight as follows:
\begin{align*} 
{\wt(\bG_m, f_{\cE_\bullet}^*\cL_{\det}|_0)}&= - \sum_{i,j} (i-j) \chi( H^*(C, \cHom(\cE_j,\cE_i))) \\
&= -2 \sum_{i>j} (i-j) \deg(\cHom(\cE_j,\cE_i))\\
&= -2 \sum_{i>j} (i-j) n_in_j(\mu(\cE_i) -\mu(\cE_j))
\end{align*}
Here we used the notation $\mu(\cE_i)=\frac{\deg(\cE_i)}{\rk(\cE_i)}$ for the slope of a vector bundle and the Riemann-Roch theorem.

In particular this number is $<0$ if $\cE$ is unstable and $\tilde{\cE}_r \subset \dots \subset \tilde{\cE}_0=\cE$ is the Harder-Narasimhan filtration of $\cE$.

We can also rewrite the weight in terms of the bundles $\tilde{\cE}_i$ as:
$$\wt(\bG_m, f_{\cE_\bullet}^* ,\cL_{\Bun_n^d}) = -2 \sum_{i=1}^r \deg(\cHom(\cE/\tilde{\cE}_i,\tilde{\cE_i})).$$
\end{example}

\begin{example}\label{wtMnd} Let us return to the case of the action of $\bG_m$ on $M_n^d$.  Note that any equivariant 2-commutative diagram
$$\xymatrix{
\Spec(k) \ar[r]^{(\cE,\theta)}\ar[d] & M_n^d \ar[d]\ar@{=>}[dl] \\
B\bG_m \ar[r] & \Bun_n^d 
}$$
induces (by definition) a map $\lambda\colon \bG_m \to \Aut(\cE)$, i.e., a grading $\cE= \oplus \cE_i$ such that the family  $(\cE,t\theta)$ on $\bG_m$ becomes isomorphic to the constant family $(\cE,\theta)$. Concretely this means $\lambda(t)^{-1} \circ \theta \circ \lambda(t) = t\theta$, so that $\theta$ must be given by a collection of morphisms $\cE_i \to \cE_{i-1} \tensor \Omega$. 

Thus for any $\bG_m$--equivariant map $f\colon \bP^1 \to M_{n}^{d}$ we have $f(0) = (\oplus \cE_i,\phi_i)$ and ${f}(\infty) = (\oplus \cE_{i}^\prime,\phi_i^\prime)$, where $\phi_i\colon: \cE_i \to \cE_{i-1}\tensor \Omega_C$ and $\phi_i^\prime\colon \cE_i^\prime \to \cE_{i-1}^\prime\tensor \Omega_C$. Thus from Example \ref{degwt} we find
$$\wt(\bG_m, \cL_{\det}|_{f(0)})= -2 \sum_{i>j} (i-j) \deg(\cHom(\cE_j,\cE_i)) = -2 \sum_{i>j} (i-j) n_in_j (\mu(\cE_i)-\mu(\cE_j))$$
$$\wt(\bG_m, \cL_{\det}|_{f(\infty)})=  -2 \sum_{i>j} (i-j) \deg(\cHom(\cE_j^\prime,\cE_i^\prime)) = -2 \sum_{i>j} (i-j) n_in_j (\mu(\cE_i^\prime)-\mu(\cE_j^\prime))$$
Since $\cL_{M_n^d}$ is positive we  have $\deg(f^*\cL_{M_{n}^{d}})= \wt(\bG_m, \cL_{\det}|_{f(0)}) - \wt(\bG_m, \cL_{\det}|_{f(\infty)})>0$, i.e. $$\wt(\bG_m, \cL_{\det}|_{f(0)}) > \wt(\bG_m, \cL_{\det}|_{f(\infty)}).$$

Since the above weight only depends on the numerical invariants $\un{n},\un{d}$ of a chain $\cE_\bullet$  we will abbreviate $$\wt(\un{n},\un{d}):=  -2\sum_{i<j} (j-i) n_in_j (\frac{d_j}{n_j}-\frac{d_i}{n_i}).$$

In particular we see that the invariant $\wt(\un{n},\un{d})$ must strictly increase on the closure of $\bG_m$-orbits of points in $F_{\un{n}}^{\un{d},-}$. It is $0$ on $F_n^d$, the open stratum parameterizing stable bundles with trivial Higgs field and stability ensures that it is $<0$ on all other strata.
\end{example}

\begin{proposition}\label{closureNilp}
If $F_{\un{m}}^{\un{e}} \cap \overline{F_{\un{n}}^{\un{d},-}} \neq \emptyset$ then we have:
$$\wt(\un{n},\un{d}) \leq  \wt(\un{m},\un{e})$$
and equality holds if and only if $(\un{m},\un{e})=(\un{n},\un{d})$.
\end{proposition}

\begin{proof}
Let $x_0\in F_{\un{m}}^{\un{e}} \cap \overline{F_{\un{n}}^{\un{d},-}}$. Choose a smooth, connected curve $g\colon X \to \overline{F_{\un{n}}^{\un{d},-}}$ such that $x_0=g(x)$ for some $x \in X$ and $X-\{x\} \subset F_{\un{n}}^{\un{d},-}$.
This defines a $\bG_m$--equivariant map $f\colon \bG_m \times X \to h^{-1}(0) \subset M_{n}^{d}$. Taking the limit $t\to \infty$ in $\bG_m$ the map $f$ extends to $$\tilde{f} \colon (\bA^1 \times X)-\{(0,x)\} \to F_{\un{n}}^{\un{d},-}$$ such that $\tilde{f}(0 \times (X-\{x\}))\subset F_{\un{n}}^{\un{d}}$ and $\tilde{f}(\bG_m \times x)=x_0$.

Since $h^{-1}(0)$ is proper there exists an equivariant blow-up $p\colon B \to \bA^1  \times X$ supported at $0\times x$, such that $\tilde{f}$ extends to a map $\tilde{f}\colon B \to M_{n}^{d}$. The exceptional fiber $p^{-1}(0\times x)$ is an equivariant chain of $\bP^1$'s and we can assume that none of the irreducible components $\bP^1\cong E_i\in p^{-1}(0\times x)$ gets contracted under $f$.

The proper transform of $0\times X$ maps to $F_{\un{n}}^{\un{d}}$ and $\bA^1 \times x_0$ maps to $x_0 \in F_{\un{m}}^{\un{e}}$.  We thus find an equvariant chain of $\bP^1$'s in $B$ that connects $x_0$ to a point which is mapped to $F_{\un{n}}^{\un{d}}$. Let us denote the fixed points of the action of $\bG_m$ on the chain of $\bP^1$'s by $p_0,\dots,p_k$ ordered in such a way that $p_0$ corresponds to the attractive fixed point $0$ on the first $\bP^1$ and $p_k$ to the repellent fixed point $\infty$ on the last line. Since we took the completion along the flow in $F_{\un{n}}^{\un{d},-}$, the standard $\bG_m$ action on $M_n^d$ will orient the fixed points $p_0,\dots,p_k$ such that $x_0=\tilde{f}(p_0)$ and $\tilde{f}(p_k) \in F_{\un{n}}^{\un{d}}$. We have seen in Example \ref{wtMnd} that the weight of the $\bG_m$ action on $\cL_{\Bun}|_{p_i}$ is strictly monotone along such a chain, i.e. if $p_0\neq p_k$ we find, 
$$\wt(\un{m},\un{e}) = \wt(\bG_m,\cL_{p_0}) > \dots > \wt(\bG_m,\cL_{p_k})=\wt(\un{n},\un{d}).$$
\end{proof}

\begin{remark}
There is at least one more necessary condition in order to have $F_{\un{m}}^{\un{e}} \cap \overline{F_{\un{n}}^{\un{d},-}} \neq \emptyset$: The universal family over $F_{\un{n}}^{\un{d},-}$ has a canonical filtration $\tilde{\cE}_\bullet\subset \cE$ any bundle in the closure of the family will admit a filtration $\overline{\cE}_\bullet$ (compatible with the Higgs field $\theta$) such that $\rk(\tilde{\cE}_i)=\rk(\overline{\cE}_i)$ and $\deg(\tilde{\cE}_i)\leq \deg(\overline{\cE}_i)$.
\end{remark}

\subsection{Finding compact cycles in fibers over reducible spectral curves}

The next step is to deform the fixed point components $F_{\un{n}}^{\un{d}}$  into cycles that do not intersect the nilpotent cone. For any $a\in \cA=\oplus_{i=1}^n H^0(C,\Omega_C^i)$ we will denote by $C_a\subset T^*C$ the corresponding spectral curve, so that as in \cite{BNR} (see \cite{Schaub} for general $a$) any Higgs bundle $(\cE,\theta)\in h^{-1}(a)$ can be viewed as a coherent $\cO_C$-torsion free sheaf on $C_a$.

Let $F_{\un{n}}^{\un{d}}$ be a non empty component with $\un{n}\neq (n)$ and chose $(\oplus \cE_i, \theta_i)\in F_{\un{n}}^{\un{d}}$.

Choose $\omega_0,\dots,\omega_r\in H^0(C,\Omega)$ such that for all $i\neq j$ the form $\omega_i-\omega_j$ has only simple zeroes and such that all of the divisors $D_{ij}:=\nst(\omega_i-\omega_j)$ are mutually disjoint.
\begin{lemma}
If $g>0$ such sections $\omega_i$ exist.
\end{lemma}
\begin{proof} This is Clifford's theorem:
The set of divisors in $|\Omega_C|$ that have only simple zeroes is non empty, e.g. because otherwise the generic multiplicities of the zeroes would define a generically injective rational map 
$$\xymatrix@C=4em{ {\bP(H^0(C,\Omega))=\bP^{g-1}} \ar@{^(-->}[r]& C^{(d_1)} \times \dots \times C^{(d_n)} \ar[r]^-{D_i \mapsto \sum iD_i} & C^{(2g-2)}. }$$
However, the dimension of the fibers of $\prod C^{(d_i)} \to \prod \Pic^{d_i}$ is smaller than $g-1$ by Clifford's theorem. Since all rational maps from $\bP^{g-1}$ to abelian varieties are constant, this contradicts the injectivity of the rational map.

The set of all $(\omega_i)$ such that for some $i\neq j$ the differential $\omega_i-\omega_j$ has multiple zeroes is thus a proper closed subset of $H^0(\Omega_C)^{r+1}$. Similarly, the subset of all $\omega$ not vanishing at any given point is a non-trivial open subset because $\dim H^0(C, \Omega(-x)) = g-2+ \dim H^0(C,\cO(x)) =g-1$, thus also the conditions $\nst(\omega_{i}-\omega_j) \cap \nst(\omega_{k}-\omega_l) =\emptyset$ define non-empty open subsets of $H^0(C,\Omega_C)^{r+1}$. 
 \end{proof}

We define $$(\cE_{\un{\omega}},\theta_{\un{\omega}}):= (\oplus \cE_i,  \oplus \theta_i + \oplus id_{\cE_i} \tensor \omega_i )$$ and similarly we will consider $$(\cE_{t\omega},\theta_{t\omega})= (\oplus \cE_i, \oplus \theta_i + \oplus id_{\cE_i} \tensor t\omega_i )$$ as a family of Higgs bundles over $C \times \bA^1$. Note that the standard action of $\bG_m$ on $\cE_i$ defines an isomorphism $(\oplus \cE_i, t( \oplus \theta_i + \oplus id_{\cE_i} \tensor \omega_i )) \cong (\oplus \cE_i, \oplus \theta_i + \oplus id_{\cE_i} \tensor t\omega_i )$.

For $i\geq j$ let us also abbreviate $\theta_{ij}:= \theta_{j-1}\circ \dots \circ \theta_i$ for $i>j$ and $\theta_{ii}=\id_{\cE_i}$.  

\begin{lemma}\label{stabledef}
\begin{enumerate}
\item[]
\item For all $t\in \bA^1$ the Higgs bundle $(\cE_{t\omega},\theta_{t\omega})$ is stable. 
\item The characteristic polynomial of $\theta_{t\omega}$ is given by $a_t:=\prod(x-t\omega_i)^{n_i}$, so that  the spectral curve is a union of the $n_i-1$-th infinitesimal neighborhoods the sections $t\omega_i$. For all $t$ the irreducible components of the spectral curve intersect at the divisors $D_{ij}=\nst{\omega_i-\omega_j}$.
\item Considered as sheaf on $C_{a_t}$ the Higgs bundle $\cE_{t\omega}$ admits a filtration $$\cF_0\subset \dots \subset \cF_{r-1} \subset \cF_r =\cE_{t\omega}$$ such that the subquotients are given by $\cF_i/\cF_{i-1} \cong \iota_{i,*} \cE_i$, where $\iota_i \colon C \to C_{a_t}$ is the closed embedding defined by $t\omega_i$. 
\item As sheaf on  $C_{a_t}$ the Higgs bundle $\cE_{t\omega}$ can also be described as:
$$\cE_{t\omega}\cong \ker\bigg( \bigoplus_{i} \iota_{i,*} \cE_i \tensor \Omega^{r-i} \to \bigoplus_{i>j} \big(\iota_{i,*} \cE_i \tensor \Omega^{r-i} \oplus \iota_{j,*} \cE_j \tensor \Omega^{r-j}\big)|_{D_{ij}}/ \phi_{ij}(\cE_i|_{D_{ij}})\bigg),$$
where $\phi_{ij} = (id, \prod \theta_{ij})$.
\end{enumerate}
\end{lemma}
\begin{proof}
(1) holds, because $\cE_{0\omega}$ is stable by assumption and stability is an open $\bG_m$--invariant condition. (2) and (3) are immediate, because the subbundles $\oplus_{j=0}^i \cE_j \subset \oplus_{j=0}^r \cE_{j}$  are Higgs subbundles of $\cE_{t\omega}$ by construction and the Higgs field induces the map $\id_{\cE_i}\tensor t\omega_i$ on the graded quotients. 

(4) follows from the eigenspace decomposition of $\theta_{t\omega}$ over $C-\cup_{i<j} D_{ij}$: Replacing $\omega_i$ by $t\omega_i$ we may assume that $t=1$ and drop the index $t$.
First note that the map 
$$ \bigoplus_{i} \iota_{i,*} \cE_i \tensor \Omega^{r-i} \to \bigoplus_{i>j} \big(\iota_{i,*} \cE_i \tensor \Omega^{r-i} \oplus \iota_{j,*} \cE_j \tensor \Omega^{r-j}\big)|_{D_{ij}}/ (id,\theta_{ij})(\cE_i \tensor \Omega^{r-i})|_{D_{ij}}$$
is surjective, because the divisors $D_{ij}$ were chosen to be disjoint, so that the restriction map $\cE_j  \to \oplus_{i>j} \cE_j|_{D_{ij}}$ is surjective. Thus the kernel of this map is a Higgs bundle of rank $n$ and degree 
$$ \sum_i d_i + n_i(r-i) (2g-2) - \sum_{i>j} n_j (2g-2) = \deg(\cE).$$

Let us define an injective Higgs bundle morphism $\Phi \colon \cE_\omega  \to\oplus \iota_{i,*} \cE_i \tensor \Omega^{r-i}$ such that the restriction of $\Phi$ to the generic point of $C$ is the inverse of the inclusion of the direct sum of the eigenspaces $\cE_i$ of $\theta_{\omega}$. Concretely define the components $\Phi_{ij}\colon \cE_j \to \iota_{i,*} \cE_i \tensor \Omega^{r-i}$ 
to be 
$$\Phi_{ij} := \theta_{ji} \tensor \prod _{k>j}(\omega_i-\omega_k)$$
for $i\leq j$ and $\Phi_{ij}=0$ otherwise.

Then we have 
\begin{align*}
\Phi(\theta(e_j))_i& = \omega_j \Phi(e_j)_i + \Phi(\theta_j(e_j))_i \\
&= \omega_j\prod _{k>j}(\omega_i-\omega_k) \theta_{ji}(e_j) +  \prod _{k>{j-1}}(\omega_i-\omega_{k}) \theta_{j-1,i} \circ \theta_j(e_j)\\
&=\omega_i \prod _{k>j}(\omega_i-\omega_k) \theta_{ji}(e_j)\\
&= \omega_i \Phi(e_j)_i
\end{align*}
So the morphism is a morphism of Higgs bundles. Moreover the image of $\Phi$ is contained in the kernel because for all $i>j$ and all $l$ we have:
\begin{align*} 
(\Phi(e_l)_j)|_{D_{ij}} &=  \big(\prod _{k>l}(\omega_j-\omega_k) \theta_{lj} (e_l)\big)|_{D_{ij}}  = \prod _{k>l}(\omega_j-\omega_k) \theta_{ij}\theta_{li} (e_l)|_{D_{ij}}\\
&\eqweil{\omega_i-\omega_j|_{D_{ij}}=0}  \prod _{k>l}(\omega_i-\omega_k) \theta_{ij}\theta_{li} (e_l) |_{D_{ij}} = \theta_{ij}(\Phi(e_l)_i )|_{D_{ij}}.
\end{align*}

\end{proof}

The description (4) of  $\cE_{t\omega}$ indicates how we can find a component of $h^{-1}(a_t)$ that contains the bundle $(\cE_{t\omega},\theta_{t\omega})$: Consider the stack
$$\prod_{i>j} \Hom (\cE_i^\univ|_{D_{ij}},\cE_j^\univ\tensor \Omega^{i-j}|_{D_{ij}}) \to \prod \Bun_{n_i}^{d_i}$$
parametrizing collections $(\cE_i; \phi_{ij})$ with $\cE_i\in \Bun_{n_i}^{d_i}$ and $\phi_{ij}\colon \cE_i|_{D_{ij}} \to \cE_j\tensor \Omega^{i-j}|_{D_{ij}}$.
The formula in (4) defines a map $$F\colon \prod \Hom (\cE_i^\univ|_{D_{ij}},\cE_j^\univ\tensor \Omega^{i-j}|_{D_{ij}})  \to h^{-1}(a_t) \subset \Higgs_{n}^d.$$

The dimension of this stack is $$\sum_i n_i^2 (g-1) + \sum_{i>j}  n_in_j (2g-2) =n^2 (g-1) = \dim h^{-1}(a_t).$$ Moreover, since the stable Higgs bundle $\cE_{t\omega}$ is contained in the image of this map, a nonempty open subset of $\prod \Hom (\cE_i^\univ|_{D_{ij}},\cE_j^\univ|_{D_{ij}})$ maps to $M_{n}^{d}$. 

\begin{lemma}\label{Pnd}
The morphism $F$ is an embedding. The image 
is the substack of those Higgs bundles $(\cE,\theta)$ that admit a filtration $(\cF_\bullet,\theta_\bullet)\subset (\cE,\theta)$ such that for all $i$ we have $\cF_i/\cF_{i-1}\in \Bun_{n_i}^{d_i}$ and the Higgs field induces $\id\tensor t\omega_i$ on this subquotient.
\end{lemma}
\begin{notation} We will denote the image $F(\prod \Hom (\cE_i^\univ|_{D_{ij}},\cE_j^\univ\tensor \Omega^{i-j}|_{D_{ij}}) \subset h^{-1}(a_t)$ by $\Pic_{\un{n},a_t}^{\un{d}}$.
\end{notation}

\begin{proof}
We will prove the statement by constructing an inverse map, defined on the image of the embedding. Again we may assume that $t=1$ and drop the index $t$. Any torsion free sheaf $\cF$ on the spectral curve $C_{a}$ admits a canonical subsheaf $$\cF_{r-1}:= \ker\big( \cF \to \iota_{r,*}(\iota_r^* \cF/\text{torsion})\big).$$
This is the unique subsheaf such that $\cF/\cF_{r-1}$ is torsion free, supported on $C_{\omega_r}$ and such that $\cF_{r-1}$ is supported on $\cup_{i=0}^{r-1} C_{\omega_i}$.
We can apply this inductively to $\cF_{r-1}$ to obtain a filtration $\cF_\bullet$ of $\cF$.

The substack of those torsion free sheaves $\cF$ on $C_a$ for which rank and degree of the $\cF_{i}$ are fixed is a locally closed substack of $h^{-1}(a)$ and it maps  
to $\prod \Higgs_{n_i}^{d_i}$ by taking subquotients $\gr_i \cF_\bullet$. Fixing the Higgs field to be equal to $\id\tensor \omega_i$ on these subquotients is a closed condition. 

Moreover the extension class of the Higgs bundles $\cF_{r-1} \to \cF \to \cE_r$ is an element in 
$$\bH^1(C,[\cHom(\cE_r,\cF_{r-1}) \map{[\quad,\theta]} \cHom(\cE_r,\cF_{r-1})\tensor \Omega]).$$
The map $[\quad,\theta]$ is an injective map of $\cO_C$ modules, because by assumption the eigenvalues of $\theta$ on $\cE_r|_{k(C)}$ and $\cF_{r-1}|_{k(C)}$ are distinct. 
The cokernel is $$\cHom\big(\cE_r,\cF_{r-1}\tensor \Omega/(\theta-\id\tensor \omega_r)(\cF_{r-1})\big).$$ Since the eigenvalues of $\theta-\omega_r$ vanish only at the divisors $D_{ir}$ we find that $$\cF_{r-1}\tensor \Omega/(\theta-\id\tensor \omega_r)(\cF_{r-1}) \cong \oplus_{i<r} (\cF_{i}/\cF_{i-1} \tensor \Omega)|_{D_{ir}}.$$
Thus we have 
\begin{align*}
[\cHom(\cE_r,\cF_{r-1}) \map{[\quad,\theta]} \cHom(\cE_r,\cF_{r-1}\tensor \Omega)] &\cong [ 0 \to \cHom(\cE_r, \cF_{r-1}\tensor \Omega/(\theta-\id\tensor \omega_r)(\cF_{r-1}))] \\
 &=[0 \to \oplus_{i<r} \cHom( \cE_r|_{D_{ir}}, \cF_i/\cF_{i-1}|_{D_{ir}} \tensor \Omega)].
\end{align*}
Thus $\Ext^1_{\Higgs}(\cE_r,\cF_{r-1}) \cong \oplus_{i<r} \Hom(\cE_{r}|_{D_{ir}},\cF_i/\cF_{i-1}\tensor \Omega|_{D_{ir}})$ and by the Yoneda description of $\Ext^1$ one sees that a collection of  homomorphisms $\phi_{ir}\colon \cE_{r}|_{D_{ir}} \to \cF_i/\cF_{i-1}\tensor \Omega|_{D_{ir}}$ corresponds to the extension given by
$$ 0 \to \cF_{r-1} \to \ker \big(\cE_r \oplus \cF_{r-1} \tensor \Omega \map{\phi-\id} \cF \tensor \Omega/ (\theta-\omega_r)(\cF_{r-1})\big) \to \cE_r\to 0.$$
Inductively the iterated extension of the $\cF_{i}/\cF_{i-1}$ is therefore canonically isomorphic to a kernel
$$ \ker( \oplus \cE_i \tensor \Omega^{r-i} \to \oplus_{i<j} \cE_i \tensor \Omega^{r-i} |_{D_{ij}}).$$
This defines an inverse map to $F$. 
\end{proof}

\begin{corollary}
The open substack $\Pic_{\un{n},a_t}^{\un{d},\st} \subset \Pic_{\un{n},a_t}^{\un{d}}$ consisting of stable bundles is dense in an irreducible component $P_{\un{n},a_t}^{\un{d}} \subseteq h^{-1}(a_t)$.
\end{corollary}
\begin{proof}
We have seen above that the stack $ \Pic_{\un{n},a_t}^{\un{d}}$ is irreducible of dimension $\dim \cA$ and by Lemma \ref{stabledef} it contains a stable point $(\cE_{t\omega},\theta_{t\omega})$. Stability is an open condition so $\Pic_{\un{n},a_t}^{\un{d},\st}$ is irreducible and of dimension $\dim \cA=\dim h^{-1}(a_t)$.
\end{proof}

Next we want to understand the closure $\overline{\bG_m \times P_{\un{n},a_1}^{\un{d}} }\subseteq M_{n}^{d}$. Since the closure defines a family over $\bA^1$ it is flat and so the intersection of the closure with $h^{-1}(0)$ is $\bG_m$ invariant, connected and of pure dimension $\dim h^{-1}(0)$. Therefore, it will be sufficient for us to determine for which $\un{m},\un{e}$ the fixed point stratum $F_{\un{m}}^{\un{e}}$ can intersect the closure:

\begin{proposition}\label{closure}
If $\overline{\bG_m \times P_{\un{n},a_1}^{\un{d}}}  \cap F_{\un{m}}^{\un{e}} \neq \emptyset$ we have $\wt(\un{n},\un{d}) \leq \wt(\un{m},\un{e})$ and equality holds if and only if $(\un{n},\un{d})=(\un{m},\un{e})$.
\end{proposition}
\begin{proof}
By construction any Higgs bundle $(\cE,\theta) \in \Pic_{\un{n},a_t}^{\un{d},\st}$ has a canonical filtration $\tilde{\cE}_0 \subset \tilde{\cE}_1 \subset \dots \subset \tilde{\cE}_r=\cE$ by Higgs subbundles such that on $\tilde{\cE}_i/\tilde{\cE}_{i-1}$ we have $\theta - \id \tensor t \omega_i=0$. Since Quot schemes are proper this implies that any $(\cE,\theta)$ in the closure $\overline{\bG_m \times P_{\un{n},a_1}^{\un{d}}}$ admits a similar filtration $\tilde{\cE}_i \subset \cE$ by Higgs subsheaves such that $\theta - \id \tensor t \omega_i=0$ holds at the generic point and therefore on the torsion free part of the subquotient.

For the saturation  $\tilde{\cE}_i^{\sat}\subset \cE$ we know that $\deg(\tilde{\cE}_i^{\sat}) \geq \deg(\tilde{\cE}_i)$ and since 
$$\wt(\un{n},\un{d}) = -2 \sum_{i=0}^{r-1} \deg( \cHom(\tilde{\cE}_i,\cE/\tilde{\cE}_i))$$
we find that $\wt(\gr(\tilde{\cE}_\bullet^{\sat})) \geq \wt(\un{n},\un{d})$ and on the graded subquotients $\tilde{\cE}_i^{\sat}/\tilde{\cE}_{i-1}^{\sat}$ we still have $\theta - \id \tensor t \omega_i=0$.

For $(\cE=\oplus \cE_i^\prime,\theta) \in  \overline{\bG_m \times P_{\un{n},a_1}^{\un{d}}}  \cap F_{\un{m}}^{\un{e}}$ we claim that we can furthermore assume that the filtration $\tilde{\cE}_i^{\sat}\subset \cE$ is compatible with the canonical $\bG_m$ action on $\oplus \cE_i^\prime$: This holds, because the action identifies $(\cE,\theta) \cong (\cE,t\theta)$ and therefore induces a $\bG_m$ action on the space of Higgs subsheaves of $(\cE,\theta)$. Since this space is proper we can again pass to the closure of a $\bG_m$ orbit in this space to find a filtration $\tilde{\cE}_j^\pprime \subset \oplus \cE_i$ such that $\tilde{\cE_j}^\pprime= \oplus_i \cE_i \cap \tilde{\cE}_j^\pprime $. 

Thus the proposition will follow from:
\begin{lemma}
Let $(\cE=\oplus_i \cE_i,\theta) \in F_{\un{m}}^{\un{e}}$ be a stable Higgs bundle, fixed under the $\bG_m$ action. Let $\tilde{\cE}_0^\pprime \subset \dots \subset \tilde{\cE}_{r^\pprime}=\cE$ be a filtration by Higgs subbundles such that  $\tilde{\cE}_j^\pprime = \oplus_i \cE_i \cap \tilde{\cE}_j^\pprime$ and such that $\theta$ induces the $0$ map on the associated graded bundles $\cE_j^\pprime:=\tilde{\cE}_j^\pprime/\tilde{\cE}_{j-1}^\pprime$.
Denote by $\un{n}^\pprime :=\rk(\gr(\tilde{\cE}_\bullet^\pprime)), \un{d}^\pprime:=\deg(\gr(\tilde{\cE}_\bullet^\pprime))$, 
then we have $$\wt( \un{m},\un{e}) \geq \wt(\un{n}^\pprime,\un{d}^\pprime)$$  and equality only holds if $\tilde{\cE_j}^\pprime = \oplus_{i=0}^j \cE_i$.
\end{lemma}
\begin{proof}
We will show that the difference of the weights is equal to the weight defined by another filtration $\cF_\bullet \subset \cE$ consisting of Higgs subbundles. Since $\cE$ is stable this weight will turn out to be negative.

Let us denote by $\tilde{\cE}_{ij}:= \cE_i \cap \cE_j^\pprime $ and by $\cE_{ij}:=\tilde{\cE}_{ij}/\tilde{\cE}_{ij-1}$ the associated graded quotients.  

Let us denote by $\gr \cE_{\bullet\bullet} := \oplus_{i,j} \cE_{ij}$ the associated graded bundle. The canonical action of $\bG_m$ on the summands $\cE_{ij}$ defines a natural map $\prod_{i,j} \bG_m \to \Aut (\gr \cE_{\bullet\bullet})$. Since the line bundle $\cL_{\det}$ defines a character  $\Aut(\gr \cE_{\bullet\bullet}) \to \bG_m= \Aut(\cL_{\det}|_{\gr \cE_{\bullet\bullet}})$ this induces a homomorphism:
$$\un{\wt} \colon \Hom(\bG_m, \prod_{i,j} \bG_m)=\bZ^{r\times r^\pprime} \to \Hom(\bG_m,\bG_m)=\bZ.$$
If we denote by $L=(l_{ij}) \in \bZ^{r\times r^\pprime}$ the matrix $l_{ij}=i$ and $C:=(c_{ij})$ with $c_{ij}=j$ we have by definition:
$$\wt(\un{m},\un{e}) - \wt(\un{n}^\pprime,\un{d}^\pprime) = \un{\wt}(L)- \un{\wt}(C) = \un{\wt}(L-C).$$
The matrix $L-C=(d_{ij})$ is given by $d_{ij} =i-j$.  We know that for all $i,j$ the Higgs field $\theta$ induces a map $\tilde{\cE}_{ij} \to \tilde{\cE}_{i-1,j-1}$ because $\theta$ maps $\cE_j^\pprime$ into $\cE_{j-1}^\pprime$ and $\cE_i $ into $\cE_{i-1}$. In particular for any $s$ the bundle $\cF_s:=\oplus_{i-j=s} \tilde{\cE}_{ij}$ is a Higgs subbundle of $\cE$ and therefore satisfies 
$$ \deg(\cHom (\oplus_{i-j=s} \tilde{\cE}_{ij}, \cE/\oplus_{i-j=s} \tilde{\cE}_{ij}) ) >0,$$
whenever $0\subsetneq \oplus_{j-i=s} \tilde{\cE}_{ij}\subsetneq \cE$ is a proper subbundle. We can rewrite this as:
\begin{align*}
 \deg(\cHom (\oplus_{i-j=s} \tilde{\cE}_{ij}, \cE/\oplus_{i-j=s} \tilde{\cE}_{ij}) ) &= \deg \cHom(\oplus_{i-j \geq s} \cE_{ij} , \oplus_{k-l<s} \cE_{kl}) >0.
\end{align*}
Thus $$\un{\wt}(L-C)= \sum_{i,j,k,l} \big((k-l)-(i-j)\big) \deg(\cHom(\cE_{ij},\cE_{kl}))$$ is the weight defined by the filtration $\cF_{-r^\prime} \subset \dots \subset \cF_{r} = \cE$. This filtration is non-trivial, unless  $\tilde{\cE}_i^\pprime = \oplus_{j=0}^i \cE_j$ and if it is a non-trivial filtration we obtain:
$$\un{wt}(L-C) = -2 \sum_{s}  \deg(\cHom (\oplus_{i-j=s} \tilde{\cE}_{ij}, \cE/\oplus_{i-j=s} \tilde{\cE}_{ij}) ) <0.$$
And therefore $\wt(\un{m},\un{e}) > \wt(\un{n}^\pprime,\un{d}^\pprime) $ as claimed.
\end{proof}
This also finishes the proof of Proposition \ref{closure}.
\end{proof}

\begin{corollary}\label{ClosureComponents}
If $F_{\un{n}}^{\un{d}}\neq \emptyset$ and $\un{n}\neq (n)$ then the intersection $\overline{\bG_m \times P_{\un{n},a_1}^{\un{d}}} \cap h^{-1}(0)$ is a union of components $\overline{F_{\un{m}}^{\un{e},-}}$.
The component $\overline{F_{\un{n}}^{\un{d},-}}$ is contained in this intersection and for $(\un{m},\un{e}) \neq (\un{n},\un{d})$ the component  $\overline{F_{\un{m}}^{\un{e},-}}$
can only occur if $\wt(\un{m},\un{e}) > \wt(\un{n},\un{d})$.
\end{corollary}

\begin{proof}
Since the map $h$ is proper and flat the closure $\overline{\bG_m \times P_{\un{n},a_1}^{\un{d}}}$ defines a proper flat family over $\bA^1$. Therefore the intersection $\overline{\bG_m \times P_{\un{n},a_1}^{\un{d}}} \cap h^{-1}(0)$  is equidimensional
of dimension $\dim  P_{\un{n},a_1}^{\un{d}} = \dim h^{-1}(0)$, so that it must be a union of irreducible components. The irreducible components of $h^{-1}(0)$ are of the form $\overline{F_{\un{m}}^{\un{e},-}}$ by Assumption $(\star_n)$.  

By Proposition \ref{closure} we know that for $(\un{m},\un{e}) \neq (\un{n},\un{d})$ the component $\overline{F_{\un{m}}^{\un{e},-}}$ can only occur if $\wt(\un{m},\un{e}) > \wt(\un{n},\un{d})$. By construction we also know that the fixed point component $F_{\un{n}}^{\un{d}}$ occurs in the closure. Since we have also seen (Proposition \ref{closureNilp}) that $F_{\un{n}}^{\un{d} }\cap \overline{F_{\un{m}}^{\un{e},-}} = \emptyset$ if   $\wt(\un{m},\un{e}) > \wt(\un{n},\un{d})$, this implies that the component $\overline{F_{\un{n}}^{\un{d},-}}$ has to be contained in $\overline{\bG_m \times P_{\un{n},a_1}^{\un{d}}} \cap h^{-1}(0)$. 
\end{proof}

\subsection{Conclusion of the argument.}

Finally, we want to translate the geometric results on the irreducible components proven so far, into results for cycle classes. Since we want to deduce results on $M_{\PGL_n}^d$ we will need to introduce some more notation. We know that $M_{\PGL_n}^d$ is the quotient of $M_n^d$ by the action of $T^*\Pic_C$ and the Hitchin base $\cA_{\PGL_n}$ is the quotient of $\cA$ by the translation action of $H^0(C,\Omega_C)$. Thus, for any $a\in \cA$ mapping to $\overline{a} \in \cA_{\PGL_n}$ the irreducible components of $h^{-1}(\overline{a})$ are the quotients of $h^{-1}(a)$ by the action of $\Pic_C$. We write $\bP P_{\un{n},a}^{\un{d}}$ and $\overline{\bP F}_{\un{n}}^{\un{d},-}$ for the components of $h^{-1}(\overline{a})$ that are the images of $P_{\un{n},a}^{\un{d}}$ and of $\overline{F}_{\un{n}}^{\un{d},-}$ respectively.

\begin{lemma}\label{cycles}
Let $\un{n},\un{d}$ be such that $\un{n}\neq (n)$ and $F_{\un{n}}^{\un{d}}\neq \emptyset$. Then we have an equality of cycle classes 
$$[P_{\un{n},a}^{\un{d}}]=[h^{-1}(0) \cap \overline{\bG_m \times \Pic_{\un{n},a}^{\un{d},\st}} ] \in H^*_c(M_n^d).$$
Moreover for any $\un{m},\un{e}$ with $\wt(\un{m},\un{e}) > \wt(\un{n},\un{d})$ there exist integers $a_{\un{m},\un{e}}\in \bN_0$ and $a_{\un{n},\un{d}} \in \bN_{> 0}$ such that
$$[P_{\un{n},a}^{\un{d}}] =a_{\un{n},\un{d}}  [\overline{F}_{\un{n}}^{\un{d},-}] + \sum_{\un{m},\un{e} \atop \wt(\un{m},\un{e})>\wt(\un{n},\un{d})}   [\overline{F}_{\un{m}}^{\un{e},-}]  \in H^*_c(M_n^d).$$

The same results hold for the classes $[\bP P_{\un{n},a}^{\un{d}}],[\overline{\bP F}_{\un{n}}^{\un{d},-}] \in H^*_c(M_{\PGL_n}^d)$.
\end{lemma}

\begin{proof}
This follows from \cite[Cycle, Th\'eor\`eme 2.3.8]{Cycle} as follows: The cycle $$P_A:=\overline{\bG_m \times P_{a_t,\un{n}}^{\un{d}}}  \subset  M_n^d \times \bA^1$$ is flat over $\bA^1$ and fiberwise of dimension $d=\dim h^{-1}(0)$. This cycle therefore defines a class $$cl(P_A,\cO_{P_A})\in H^0(\bA^1, \bR^d p_{\bA_1,!} \bQ_\ell) \cong H^d_c(M_n^d),$$ that commutes with every base change, if one uses the derived pull back for $\cO_{P_A}$. Since $P_A\to \bA^1$ is flat this implies that the classes of all fibers relative to the structure sheaf $cl(p^{-1}(a),\cO_{p^{-1}(a)})$ coincide. For $a=1$ we obtain the class $[P_{\un{n},a}^{\un{d}}]$ and for $a=0$ we obtain the class of $[P_A \cap h^ {-1}(0)]$. We have seen in Corollary \ref{ClosureComponents} that this intersection is a union of irreducible components $\overline{F_{\un{m}}^{\un{e},-}}$ with $\wt(\un{m},\un{e}) < \wt(\un{n},\un{d})$ and the component $\overline{F_{\un{n}}^{\un{d},-}}$. 

The same reasoning holds for the cycles $\bP P_{\un{n},a}^{\un{d}}$ and $\overline{\bP F}_{\un{n}}^{\un{d},-} \in M_{\PGL_n}^d$.
\end{proof}

From this Lemma we can now deduce the vanishing of the intersection form on $H^*_c(M_{\PGL_n}^d)$:
\begin{proof}[Proof of Theorem \ref{main}]
Since we assumed that the varieties $\overline{\bP F}_{\un{n}}^{\un{d},-}$ are irreducible, these varieties are the irreducible components of $h^{-1}_{\PGL_n}(0)$. In particular, the classes $[\overline{\bP F}_{\un{n}}^{\un{d},-}] \in H^*_c(M_{\PGL_n}^d)$ generate the middle cohomology group. Thus Lemma \ref{cycles} shows that the classes $[P_{\un{n},a}^{\un{d}}]$ together with the class $[\bP F_{n}^{d}] = [\Bun_{PGL_n}^{d,\st}]$ also generate $H^{\dim M_{\PGL_n}^d}_c(M_{\PGL_n}^d)$. 

Lemma \ref{cycles} implies that for any $(\un{n},\un{d}),( \un{m},\un{e})$ with $\un{n}\neq(n), \un{m}\neq (n)$ we have 
$$[\bP P_{a,\un{n}}^{\un{d}}] \cup [\bP P_{a,\un{m}}^{\un{e}}] =0 \in H_c^*(M_{\PGL_n}^d),$$ because the classes are independent of the choice of $a$ and for different $a$ the cycles are contained in different fibers of the Hitchin map, so these cycles do not intersect geometrically and thus (\cite[Cycle, Remarque 2.3.9]{Cycle}) the cup-product of their classes is $0$. Similarly $$[{\bP P}_{a,\un{n}}^{\un{d}}]\cup [\bP F_{n}^{d}]= 0$$ because $\bP F_{n}^{d} = \Bun_{PGL_n}^{d,\st} \subset h_{\PGL_n}^{-1}(0)$. 

Finally by the Poincar\' e--Hopf theorem 
$$[\bP F_{n}^d] \cap [\bP F_{n}^d]= [\Bun_{\PGL_n}^{d,\st}]\cap [\Bun_{\PGL_n}^{d,\st}] = \chi ( \Bun_{\PGL_n}^{d,\st} )$$
and $\chi ( \Bun_{\PGL_n}^{d,\st} ) = 0$, e.g. by the explicit formula for the cohomology found by Harder--Narasimhan and Atiyah--Bott \cite[Theorem 10.10]{AB}. Therefore the intersection form vanishes on $H^{\dim M_{\PGL_n}^d}_c(M_{\PGL_n}^d)$ and since $H^{*}_c(M_{\PGL_n}^d)=0$ for $*<\dim M_{\PGL_n}^d$ this shows that the intersection form vanishes in all degrees.
\end{proof}

\section{Irreducibility of the  spaces of stable chains} 

In this section we want to prove Theorem \ref{irreducible}, i.e., we show that the coarse moduli spaces $\Chain_{\un{n}}^{\un{d},\alpha-\sst}$ of $\alpha-$semistable chains are irreducible if $\alpha$ satisfies $\alpha_{i+1}-\alpha_i>2g-2$ for all $i$  and $\alpha$ is not a critical value. Let us begin by recalling these notions from \cite{AGPS}.

\subsection{Basic results on stability of chains}
 For $\alpha\in \bR^{r+1}$ the $\alpha$-slope of a chain $\cE_\bullet$ is defined as $$\mu_{\alpha}(\cE_\bullet) := \frac{\deg(\oplus \cE_i)}{\rk(\oplus {\cE_i})} + \frac{\sum_{i=0}^r \alpha_i \rk(\cE_i)}{\rk(\oplus \cE_i)}.$$
Since the $\alpha$-slope only depends on the rank $\un{n}:=(\rk(\cE_i))$ and the degree $\un{d}:=(\deg(\cE_i))$ we often write
$$ \mu_{\alpha}(\un{n},\un{d}) := \frac{ \sum_{i=0}^r d_i + \alpha_i n_i}{\sum_{i=0}^r n_i}.$$
A chain $\cE_\bullet$ is called (semi-)stable if for all proper subchains $0 \subsetneq \cF_\bullet \subsetneq \cE_\bullet$ we have $\mu_{\alpha}(\cF_\bullet) \leq \mu_{\alpha}(\cE_\bullet)$.

A stability parameter $\alpha\in \bR^{r+1}$ is called {\em critical} for some rank and degree $(\un{n},\un{d})$ if there exist $\un{n}^\prime,\un{d}^\prime$ with $\un{n}^\prime < \un{n}$ (i.e. $n_i^\prime \leq n_i$ and $\sum n_i^\prime < \sum n_i)$ and $\gamma \in \bR^{r+1}$ such that $\mu_{\alpha}(\un{n}^\prime ,\un{d}^\prime) = \mu_{\alpha}(\un{n},\un{d})$ and $\mu_{\gamma}(\un{n}^\prime ,\un{d}^\prime) \neq \mu_{\gamma}(\un{n},\un{d})$.

For given $\un{n}^\prime,\un{d}^\prime$ the condition that $\alpha\in \bR^{r+1}$ satisfies $\mu_{\alpha}(\un{n}^\prime ,\un{d}^\prime) = \mu_{\alpha}(\un{n},\un{d})$ is a linear equation in $\alpha$.  If this equation is non-trivial, then the corresponding a hyperplane in $\bR^{r+1}$ is called a wall. The union of the walls is known to be closed \cite[Section 2.4]{AGPS}.

The Higgs bundle defined by a chain is the bundle $\cE^\prime := \oplus \cE_i \tensor \Omega_C^{-r+i}$ together with the Higgs field $\theta := \oplus \phi_i$. It is semistable if and only if the chain is semistable with respect to the parameter $\alpha_{\Higgs}=(\alpha_i)_{i}=(i(2g-2))_{i}$.   For any $\alpha\in \bR^{r+1}$ we will abbreviate the  condition $\alpha_{i+1}-\alpha_i>2g-2$ for all $i$ as $$\alpha>\alpha_{\Higgs}.$$

\begin{remark}\label{stab}
If $(n,d)=1$ then for all $\un{n},\un{d}$ occurring as rank and degree of fixed point components in the Hitchin moduli space $M_n^d$, the stability parameter ${\alpha}_{\Higgs}$ is not a critical value and therefore $\Chain_{\un{n}}^{\un{d},\alpha_{\Higgs}-\sst}$ coincides with a moduli space of $\alpha$-stable chains for some $\alpha$ satisfying $\alpha_{i+1}-\alpha_i>2g-2$ for all $i$. Thus Theorem \ref{irreducible} implies that Assumption \ref{assum1} from Section 1 holds. (See \cite[Section 7]{GPHS} for a similar discussion.)
\end{remark}

If a chain $\cE_\bullet$ is not $\alpha$-semistable, then it has a canonical Harder-Narasimhan filtration $\cF_\bullet^1 \subsetneq \dots \subsetneq \cF_\bullet^h=\cE_\bullet$ such that $\cF_\bullet^i/\cF_\bullet^{i-1}$ are $\alpha$-semistable and $\mu_\alpha (\cF_\bullet^1) > \mu_{\alpha}(\cF_\bullet^2) > \dots > \mu_{\alpha}(\cF_\bullet^h)$. For an unstable chain $\cE_\bullet$ we will denote by $\cE_\bullet^i := \cF_\bullet^i/\cF_\bullet^{i-1}$ the graded quotients of its Harder-Narasimhan filtration.

The collection $t=(\rank(\cF_\bullet^i/\cF_\bullet^{i-1}), \deg(\cF_\bullet^i/\cF_\bullet^{i-1}))$ is called the type of the Harder--Narasimhan filtration. The locally closed substacks of unstable chains of a given type $t$ are called Harder-Narasimhan strata and we denote them by $\Chain_{\un{n}}^{\un{d},\alpha,t}$.

\subsection{The wall-crossing argument}
To prove the theorem we use the same method as in as in \cite{BGPG}, varying $\alpha$ and estimating the dimension of the part of the moduli space that changes when $\alpha$ crosses a wall. 

Let us begin by recalling the wall-crossing argument in the language of algebraic stacks in this situation (\cite[Section 3]{GPH} and the references therein): 
Given a chain $\cE_\bullet=(\cE_i,\phi_i)$ deformations of $\cE_\bullet$ are parametrized by the cohomology of the complex
$$ [\oplus \cHom(\cE_i,\cE_i) \map{b=\oplus [\quad,\phi_i]} \oplus \cHom(\cE_i,\cE_{i-1})].$$
If  $\alpha>\alpha_{\Higgs}$ we know that the $\bH^2$ of the complex vanishes, so that deformation theory for $\alpha$-semistable chains is unobstructed (see \cite[Lemma 4.6]{GPHS}\footnote{Unfortunately this Lemma is only formulated for stable chains, because it allows the case $\alpha=\alpha_{\Higgs}$. However for $\alpha>\alpha_{\Higgs}$ the first inequality in the proof is always strict, so that the vanishing of $H^2$ also holds for strictly semistable chains if $\alpha>\alpha_{\Higgs}$.} ). Therefore we have: 
$$\dim \Chain_{\un{n}}^{\un{d},\alpha-\sst} = - \chi( \bH^*( C, [\oplus \cHom(\cE_i,\cE_i) \map{b=\oplus [\quad,\phi_i]} \oplus \cHom(\cE_i,\cE_{i-1})])).$$
and this stack is smooth.

Given chains $\cE_\bullet^\prime, \cE_\bullet^\pprime$ we will abbreviate:
$$ \chi(\cE_\bullet^\prime, \cE_\bullet^\pprime):= \chi(  \bH^*( C, [\oplus \cHom(\cE_i^\prime,\cE_i^\pprime) \map{b=\oplus [\quad,\phi_i]} \oplus \cHom(\cE_i^\prime,\cE_{i-1}^\pprime)] )).$$

If $\alpha_t := \alpha + t\delta$ is a family of parameters such that  $\alpha_{0}$ is a critical value and such that $\alpha_t > \alpha_{Higgs}$ for all $t$ in some neighborhood of $0\in \bR$, then we know from \cite[Proposition 2] {GPH} that for $t\in (0,\epsilon)$ with $\epsilon$ sufficiently small and $\alpha^+:=\alpha_t, \alpha^-:=\alpha_{-t}$ we have:  
$$ \Chain_{\un{n}}^{\un{d}, \alpha_0-\sst} = \Chain_{\un{n}}^{\un{d},\alpha^+-\sst} \cup \bigcup_{t\in I^+}  \Chain_{\un{n}}^{\un{d},\alpha^+,t},$$
$$ \Chain_{\un{n}}^{\un{d}, \alpha_0-\sst} = \Chain_{\un{n}}^{\un{d},\alpha^--\sst} \cup \bigcup_{t\in I^-} \Chain_{\un{n}}^{\un{d},\alpha^-,t}.$$
Here $I^+$ (resp. $I^-)$ is the finite set of types $(\un{n}^i,\un{d}^i)$ of $\alpha^+$-Harder-Narasimhan strata such that $\mu_{\alpha_0}(\un{n}^i,\un{d}^i) = \mu_{\alpha_0}(\un{n},\un{d})$ for all $i$.
Moreover a type $t=((\un{n}^0,\un{d}^0),\dots, (\un{n}^k,\un{d}^k)) $ of HN-strata occurs in $I^+$  if and only if the opposite type $t^{\opp}:=((\un{n}^k,\un{d}^k),\dots, (\un{n}^0,\un{d}^0))$ occurs in $I^-$.

\begin{lemma}
Let $\alpha_t \in \bR^{r+1} $ be a family of stability parameters as above, such that $\Chain_{\un{n}}^{\un{d},\alpha^+-\sst}$ is irreducible and $\Chain_{\un{n}}^{\un{d},\alpha^--\sst}$ is not irreducible. Then there exists a type $t=(\un{n}^i,\un{d}^i) \in I^+$ such that for all $\cE_\bullet \in  \Chain_{\un{n}}^{\un{d},\alpha^+,t}$ we have $$\chi(\cE_\bullet^l, \cE_\bullet^{j}) =0 \text{ for all } l<j,$$ where the $\cE_\bullet^i$ are the graded quotients of the Harder-Narasimhan filtration of $\cE_\bullet$.    
\end{lemma}
\begin{proof}
The dimension of a HN--stratum of type $t=(\un{n}^\bullet,\un{d}^\bullet)$ can be computed as follows: Let $\cE_\bullet$ be a chain of this type, $\cE^{i}_\bullet:=\cF_\bullet^i/\cF_\bullet^{i-1}$ be the graded quotients of the Harder-Narasimahn filtration of the chain $\cE_\bullet$.

Then the dimension of the stratum is (\cite[Proposition 4.8]{GPHS}):
\begin{align*} 
\dim \Chain_{\un{n}}^{\un{d},\alpha^+,t}&= - \sum_{j} \chi( \cE^j_\bullet,\cE^j_\bullet) - \sum_{l>j} \chi( \cE^l_\bullet,\cE^j_\bullet)\\
\end{align*}

and the dimension of the opposite stratum is given by the same expression, where the sum is replaced by a sum over $l<j$:
\begin{align*} 
\dim \Chain_{\un{n}}^{\un{d},\alpha^-,t^{\opp}}&= - \sum_j \chi( \cE^j_\bullet,\cE^j_\bullet) - \sum_{l<j} \chi( \cE^l_\bullet,\cE^j\bullet).\\
\end{align*}

Finally the dimension of $\Chain_{\un{n}}^{\un{d},\alpha-\sst}$ is given by 
\begin{align*} 
\dim \Chain_{\un{n}}^{\un{d},\alpha-\sst}&=- \sum_{l,j} \chi(\cE^l_\bullet,\cE^j_\bullet)\\
\end{align*}
We know moreover that $\Hom(\cE^l_\bullet, \cE_\bullet^j)=0$ for all $l\neq j$ in our situation since the two chains are semistable chains and either $\mu(\alpha^+)(\cE_\bullet^l) > \mu(\alpha^+)(\cE_\bullet^j)$ or $\mu(\alpha^-)(\cE_\bullet^l) > \mu(\alpha^-)(\cE_\bullet^j)$. Therefore the groups $\bH^0$ and $\bH^2$ of the complex $$ [\oplus \cHom(\cE^l_i,\cE^j_i) \map{b=\oplus [\quad,\phi_i]} \oplus \cHom(\cE^l_i,\cE^j _{i-1})]$$
vanish for $l\neq j$. Thus for all $l\neq j$:
$$\chi(\cE^l_\bullet,\cE^j_\bullet) \leq 0.$$
Therefore an $\alpha^+$-HN-stratum can only have dimension equal to the dimension of the stack of all chains if
$$ \chi(\cE^l_\bullet,\cE^j_\bullet) =0$$
for all $l<j$.
\end{proof}

For $\cE_\bullet$ as in the above Lemma we know by \cite[Proposition 4.5]{AGPS} that for all $l,j$ $$ \chi( \bH^*( [\oplus \cHom(\cE^l_i,\cE^j_i) \map{b=\oplus [\quad,\phi_i]} \oplus \cHom(\cE^l_i,\cE^j _{i-1})])) <0$$
if $b$ is not generically an isomorphism. Note that in  \cite[Proposition 4.5]{AGPS} the chains $\cE_\bullet^l,\cE_\bullet^i$ were assumed to be polystable, however, for strictly semistable chains the above complex admits a filtration, such that the subquotients are given by the analogous complex for the stable subquotients of $\cE^l_\bullet,\cE_\bullet^j$. Since $\chi$ is additive with respect to filtrations we see that we must have $ \chi(\cE_\bullet^l,\cE_\bullet^j)<0$ unless all graded pieces of $b$ are generically isomorphisms.

If $b$ is generically an isomorphism, but not an isomorphism, then the complex $$[\oplus \cHom(\cE^l_i,\cE^j_i) \map{b=\oplus [\quad,\phi_i]} \oplus \cHom(\cE^l_i,\cE^j _{i-1})]$$ 
is isomorphic to its cokernel which then must be a torsion sheaf on $C$, so also in this case its Euler characteristic is $<0$. Thus we know:

\begin{corollary}[{\cite[Proposition 4.5]{AGPS}}]
The Euler charactersitic 
$$\chi( H^*( [\oplus \cHom(\cE^l_i,\cE^j_i) \map{b=\oplus [\quad,\phi_i]} \oplus \cHom(\cE^l_i,\cE^j _{i-1})]))$$
can only vanish if the map $b$ is an isomorphism.
\end{corollary}

In view of this corollary we will say that a type $t=(\un{n}^\bullet,\un{d}^\bullet)$ of Harder-Narasimhan is {\em maximal} if for all chains $\cE_\bullet\in \Chain_{\un{n}}^{\un{d},\alpha,t}$ the morphisms 
$$[\oplus \cHom(\cE^l_i,\cE^j_i) \map{b=\oplus [\quad,\phi_i]} \oplus \cHom(\cE^l_i,\cE^j _{i-1})]$$
are isomorphisms for all $l>i$. 

Our next aim will be to determine the possible types of maximal Harder-Narasimhan strata.

\subsection{Maximal Harder--Narasimhan strata}
Let us call a pair of chains $\cE_\bullet^\prime,\cE_\bullet^\pprime$ a {\em maximal pair} if the map
$$ b\colon\oplus \cHom(\cE^\pprime_i,\cE^\prime_i) \map{\oplus [\quad,\phi_i]} \oplus \cHom(\cE^\pprime_i,\cE^\prime _{i-1})$$
is an isomorphism.

To state our characterization of maximal pairs let us introduce a notation. Given a chain $\cE_\bullet=(\cE_j,\phi_j)$ and an index $i$ we will denote by
$\cE_{\widehat{i}}$ the chain obtained by removing the $i$-th bundle and composing $\phi_{i+1}$ and $\phi_i$, i.e.,
$$\cE_{\widehat{i}}= ( \cE_r \to \dots \to \cE_{i+1} \map{\phi_i \circ\phi_{i+1}}  \cE_{i-1} \to \dots \to \cE_0).$$

\begin{lemma}\label{maximalPairs}
Let $\cE_\bullet^\pprime$, $\cE_\bullet^\prime$  be a maximal pair of chains of ranks $\un{n}^\pprime$ and $\un{n}^\prime$.
Then the following hold: 
\begin{enumerate}
\item There exists  $0\leq i\leq r$ such that either $\cE_i^\pprime=0$ or $\cE_i^\prime=0$.
\item If for some $i$ the morphism $\phi_i^\prime$ is an isomorphism, then the pair of chains $\cE_{\widehat{i}}^\pprime$, $\cE_{\widehat{i}}^\prime$  is also a maximal pair.
\item If for some $i$ the morphism $\phi_i^\pprime$ is an isomorphism, then the pair of chains $\cE_{\widehat{i-1}}^\pprime$, $\cE_{\widehat{i-1}}^\prime$ is also a maximal pair.	
\item If the length $r$ of the chain is at least $2$, then there exists an $i$ such that $\phi_i^\prime$ or $\phi_i^\pprime$ is an isomorphism. 
\item If $r=1$ then one of the morphisms $\phi_1^\prime,\phi_1^\pprime$ is an isomorphism, or $\cE_1^\pprime=0=\cE_0^\prime$.
\end{enumerate}
\end{lemma}

\begin{proof}
We will prove the Lemma by induction on the length of the chain. For $r=0$ there is nothing to prove.

It will be useful to extend our chains by $0$ for $\bullet=-1$ and $\bullet=r+1$ i.e., to denote $\cE_{-1}^\prime:=0,\cE_r^\prime:=0$, $\phi_0:=0,\phi_{r+1}=0$ and similarly for $\cE_\bullet^\pprime$.

If $$ b\colon\oplus \cHom(\cE^\pprime_i,\cE^\prime_i) \map{\oplus [\quad,\phi_i]} \oplus \cHom(\cE^\pprime_i,\cE^\prime _{i-1})$$
is an isomorphism, so is the dual map 
$$ b^\vee \colon\oplus \cHom(\cE^\prime_{i-1},\cE^\pprime_i) \map{\oplus [\quad,\phi_i]} \oplus \cHom(\cE^\prime_i,\cE^\pprime _{i}).$$
Thus, the pair of chains $\cE_\bullet^\pprime,\cE_{\bullet-1}^\prime$ will also be a maximal pair.

Next, let us collect some basic constraints on the maps $\phi_i^\prime,\phi_i^\pprime$.
If $b$ is an isomorphism, then for $i=0,\dots,r+1$ we know that the map
$$ \cHom(\cE_i^\pprime,\cE_i^\prime) \oplus \cHom(\cE^\pprime_{i-1},\cE_{i-1}^\prime) \map{\phi_i^\prime \circ \un{\quad} + \un{\quad}\circ \phi_{i}^\pprime} \cHom(\cE_i^\pprime,\cE_{i-1}^\prime)$$
must be surjective. 
The cokernel of the map $\phi_i^\prime \circ \un{\quad}$ is $ \cHom(\cE_i^\pprime,\coker(\phi_i^\prime))$, so the induced map
$$ \cHom(\cE^\pprime_{i-1},\cE_{i-1}^\prime) \to  \cHom(\cE_i^\pprime,\coker(\phi_i^\prime))$$
must be surjective. If $\coker(\phi_i^\prime) \neq 0$ the morphism $\phi_{i}^\pprime$ must be injective at the points in the support of the cokernel.
Thus $\phi_i^\prime$ is surjective or $\phi_{i}^\pprime$ must be injective.

The same argument applies to the pair $\cE_\bullet^\pprime,\cE_{\bullet-1}^\prime$, thus $\phi_{i}^\pprime$ is surjective or $\phi_{i-1}^\prime$ is injective.

In particular putting $i=1$ we find that $\cE_0^\prime=0$ or $\phi_1^\pprime$ is surjective and for $i=r+1$ we see that $\cE_r^\pprime=0$ or $\phi_r^\prime$ is injective.

This already allows to prove (4) and (5): 

If $\cE_0^\prime\neq 0$ the morphism $\phi_1^\pprime$ is surjective. By the first constraint above $\phi_1^\pprime$ is also injective, or $\phi_1^\prime$ is surjective as well. 
Inductively, we therefore either find that one of the maps $\phi_i^\prime, \phi_i^{\pprime}$ is an isomorphism, or we find that for all $i=1,\dots,r+1$ the maps are surjective, but then both chains are $0$.

If $\cE_0^\prime=0$, then the truncated pair of chains $\cE^\pprime_{\bullet>0},\cE^\prime_{\bullet>0}$ is again maximal. If $r>1$ we can then either conclude that $\cE_{1}^\prime=0$ so that $\phi_1^\prime$ is an isomorphism, or that by induction one of the morphisms $\phi_i^\prime, \phi_i^\pprime$ for $i>1$ is an isomorphism. This shows (4). 
For $r=1$ we conclude that $\cE_1^\prime=0$ or $\cE_1^\pprime=0$, showing (5).

(2) and (3) are easy:  

If one of the maps $\phi_i^\pprime\colon \cE_{i}^\pprime \to \cE_{i-1}^\pprime$ is an isomorphism then the map $$\underline{\quad}\circ \phi_i^\pprime\colon \cHom(\cE_{i-1}^\pprime,\cE_{i-1}^\prime) \to \cHom(\cE_{i}^\pprime,\cE_{i-1}^\prime)$$ is an isomorphism.

In particular the acyclic complex $\cHom(\cE_{i-1}^\pprime,\cE_{i-1}^\prime) \to \cHom(\cE_{i}^\pprime,\cE_{i-1}^\prime)$ maps injectively into $$[\oplus \cHom({\cE}^\pprime_i,{\cE}^\prime_i) \map{b=\oplus [\quad,\phi_i]} \oplus \cHom({\cE}^\pprime_i,{\cE}^\prime _{i-1})])$$
and the quotient is the complex defined by the pair $\cE_{\widehat{i-1}}^\pprime$, $\cE_{\widehat{i-1}}^\prime$.

Analogously if one of the maps $\phi_i^\prime\colon \cE^\prime_{i} \to \cE^\prime_{i-1}$ is an isomorphism, then the map  $$\cHom(\cE_{i}^\pprime,\cE_{i}^\prime) \to \cHom(\cE_{i}^\pprime,\cE_{i-1}^\prime)$$ is an isomorphism, and we can proceed as above, this time removing the $i$th entry of the two chains.

Now (1) follows, because this holds for $r=0$ and $r=1$ and by (4),(5) there always exists an $i$ such that one of the morphisms occurring in the chains is an isomorphism and this allows to shorten the chain by (2) and (3), so that the claim follows by induction.
\end{proof}

\begin{corollary}\label{maximalSummands}
Let $\cE_\bullet^\pprime$, $\cE_\bullet^\prime$  be a maximal pair of chains of ranks $\un{n}^\pprime$ and $\un{n}^\prime$.
Suppose that the sets $I^\prime :=\{i | n_i^\prime \neq 0\}$,$I^\pprime:=\{i | n_i^\pprime \neq 0\}$  and $I^\prime\cup I^\pprime=[0,r]\cap \bZ$ are strings of consecutive integers.
Then one of the following holds:
\begin{enumerate}
\item There exists $0\leq j\leq r$ such that $I^\prime=\{ i | i>j \}$ and $I^\pprime = \{ i | i \leq j\}$ 
\item $\cE_r^\prime=0$. Moreover in this case let $k$ be minimal, such that $\cE_i^\prime=0$ for $k\leq i$ and $l$ be the maximal integer such that $\cE_i^\prime=0$ for $i<l$. Then $\cE^\pprime_\bullet$ contains a direct summand isomorphic to the chain 
$$\cE_r^\pprime \to \dots \to \cE_k^\pprime =\dots =\cE_k^\pprime \to \cE_{l-1}^\pprime \to \dots \cE_0^\pprime.$$
\item $\cE_0^\pprime=0$. Moreover in this case let $l$ be maximal , such that $\cE_i^\pprime=0$ for $i\leq l$ and $k$ be the minimal integer such that $\cE_i^\pprime=0$ for $i>k$. Then $\cE^\prime_\bullet$ contains a direct summand isomorphic to the chain 
$$\cE_r^\prime \to \dots \to \cE_{k+1}^\prime \to \cE_l^\prime =\dots =\cE_l^\prime \to \cE_{l-1}^\prime \to \dots \cE_0^\prime.$$
\end{enumerate}
\end{corollary}
\begin{proof}
Let us first prove that one of the given conditions occurs:  
From Lemma \ref{maximalPairs} we know that there exists an $i$ such that one of the bundles $\cE_i^\prime, \cE_i^\pprime$ is $0$. If some $\cE_i^\prime=0$, let us choose $i$ to be maximal with this property. By our assumption $\cE_j^\prime=0$ either for all $j\geq i$ in which case $\cE_r^\prime=0$, which is case (2),  or $\cE_j^\prime=0$  for all $j \leq i$. In this case the pair $\cE^\pprime_{\bullet>i},\cE_{\bullet>i}^\prime$ is again maximal so again one of the bundles occurring in this pair must be $0$, so by construction there now exists a minimal $j>i$  with $\cE_j^\pprime=0$.  Since $I^\prime\cup I^\pprime=[0,r]$ we must then have $\cE_k^\pprime =0$ for all $k\geq j$, but in this case the truncated pair of chains  $\cE^\pprime_{j>\bullet>i},\cE_{j>\bullet>i}^\prime$ will again be maximal and contains only non-zero bundles, so $j=i+1$ and the chain is of the type described in (1).

The same argument shows that if some $\cE_i^\pprime=0$ then either $\cE_0^\pprime=0$ or the pair is of the type described in (1). Thus one of the conditions listed in the corollary holds.

Suppose now  that $\cE_r^\prime=0$ and write $I^\prime =[m,k-1]$ with $r\geq k>l\geq 0$. To show that $\cE_\bullet^\pprime$ contains a direct summand of the given form we proceed by induction on $r$. For $r=0$ there is nothing to show and for $r=1$ Lemma \ref{maximalPairs} shows that $\phi_1^\pprime$ must be an isomorphism. 

In general, we may reduce to the case $[m,k-1]=[0,r-1]$, because a direct summand of the truncated chain $\cE^\pprime_{m\leq \bullet\leq k}$ that contains $\cE_k^\pprime$ defines one of $\cE_\bullet^\pprime$ and the maximality of the pair also holds for the truncated chain since $\cE_i^\prime=0$ for $i\not\in I^\prime$.

By Lemma \ref{maximalPairs} there exists $i$ such that either $\phi_i^\prime$ or $\phi_i^\pprime$ is an isomorphism. In the first case the pair $\cE_{\widehat{i}}^\pprime$, $\cE_{\widehat{i}}^\prime$ is again maximal and again satisfies that the bundle with the largest index in  $\cE_{\widehat{i}}^\prime$ is $0$, so by induction $\cE_{\widehat{i}}^\pprime$  contains a direct summand of the form described in (2).

Also we know $i<r$ because we assumed $\cE_{r-1}^\prime\neq 0$ so that $\phi_r^\prime$ is not an isomorphism.

Since we found a direct summand of the form given in (2) in the chain $\cE_{\widehat{i}}^\pprime$ we know that the composition $\phi_{r0}^\pprime$ is an injective, so that the map $\phi_{ri}^\pprime\colon \cE_r^\pprime \to \cE_i^\pprime$ must be injective and composing this map with $\phi_{i}$ we find a splitting of the inclusion of the subchain 
$$\cE_r^\pprime = \dots = \cE_r^\pprime $$
of $\cE_\bullet^\pprime$. 

Similarly if $\phi_i^\pprime$ is an isomorphism the pair $\cE_{\widehat{i-1}}^\pprime$, $\cE_{\widehat{i-1}}^\prime$ is maximal and satisfies $\cE_r^\prime=0$, so again we find a direct summand in $\cE_{\widehat{i-1}}^\pprime$, which extends to a direct summand in $\cE_\bullet$, because direct summands always extend if one adds an isomorphism into a chain.

If $\cE_0^\pprime=0$ an argument dual to the above gives the result.
\end{proof}

The summands occurring in Corollary \ref{maximalSummands} are isomorphic to one of the canonical subchains or quotients used in \cite[Proposition 4]{GPH} to give necessary conditions for the existence of semistable chains. Since we will use them let us briefly recall the definition of these chains:

Given a chain $\cE_\bullet$ and a stability parameter $\alpha\in \bR^{r+1}$ satisfying $\alpha_{i+1} > \alpha_i$ for all $i$ we will call the following chains the {\em canonical test chains} for the existence of semistable chains:
\begin{enumerate}
\item For $0\leq i <r$ the chain $\cE_{\bullet\leq i} :=(0\to \dots \to 0 \to \cE_i \to \cE_{i-1} \to \dots \cE_0)$.
\item If there exist $0\leq l<k \leq r$ such that for all $i$ with $l\leq i <k$ we have $n_i \geq n_k$, the chain 
$$ S_{[lk]}(\cE_\bullet) = (\cE_r \to \dots \to \cE_k \map{id} \dots \map{id} \cE_k \to \cE_{l-1} \to \dots \to \cE_0).$$
\item If there exist $0\leq l<k \leq r$ such that for all $i$ with $l\leq i <k$ we have $n_i \geq n_l$, the chain
 $$ Q_{[lk]}(\cE_\bullet)=(\cE_r \to \dots \to \cE_{k+1} \to \cE_l \map{id} \dots \map{id} \cE_l \to \cE_{l-1} \to \dots \to \cE_0).$$
\end{enumerate}
The chains $\cE_{\bullet\leq i}$ are canonical subchains of $\cE_\bullet$, the chains $ S_{[lk]}(\cE_\bullet)$ map to $\cE_\bullet$ and $\cE_\bullet$  maps to the chains $Q_{[lk]}(\cE_\bullet) $.
By \cite[Proposition 4]{GPH} we know that $\alpha$-semistable chains can only exist if:
\begin{enumerate}
\item  $\mu_\alpha( \cE_{\bullet\leq i} ) \leq \mu_{\alpha}(\cE_\bullet)$, 
\item $\mu_\alpha( S_{[lk]}(\cE_\bullet) ) \leq \mu_{\alpha}(\cE_\bullet)$ and
\item $ \mu_{\alpha}(\cE_\bullet) \leq \mu_{\alpha}(Q_{[lk]}(\cE_\bullet))$
\end{enumerate}
for all possible choices of $i,l,k$.

\subsection{Proof of Theorem \ref{irreducible}}	

Let us denote by $\bR^{r+1}_{>\Higgs}\subset \bR^{r+1}$ the subset defined by  $\alpha>\alpha_{\Higgs}$. The irreducibility of moduli spaces of $\alpha$-semistable chains will follow from the following Proposition:
\begin{proposition}\label{path}
For any $\un{n},\un{d}$ and $\alpha\in \bR^{r+1}_{>\Higgs}$ there exists $\alpha_\infty\in \bR^{r+1}_{>\Higgs}$ and a path $\gamma\colon \alpha \to \alpha_\infty$ that is linear in a neighborhood of any critical point lying on $\gamma$, such that 
\begin{enumerate}
\item if $\un{n}=(m,\dots,m)$ is constant then $\Chain_{\un{n}}^{\un{d},\alpha_\infty-\sst}$ is irreducible and for none of the critical values in $\gamma$ a maximal HN-stratum occurs in the wall crossing decomposition.
\item if $\un{n}$ is not constant then $\alpha_\infty$ is a critical value, such that $\Chain_{\un{n}}^{\un{d},\alpha_\infty^+}=\emptyset$  and $\Chain_{\un{n}}^{\un{d},\alpha_\infty}$ contains a unique maximal $\alpha^-_\infty$-HN-stratum. Moreover for none of the other critical values in $\gamma$ maximal HN-strata occur in the wall crossing decomposition.
\end{enumerate}
\end{proposition}
\begin{proof}
By \cite[Lemma 8]{GPH} for non-constant $\un{n}$ there exists a line in $\bR^{r+1}_{>\Higgs}$ such that $\Chain_{\un{n}}^{\un{d},\alpha_\infty^+-\sst}$ is empty. Since for given $\un{n},\un{d}$ the walls defined by the critical parameters define a locally finite partition of $\bR^{r+1}$ \cite[Section 2.4]{AGPS} we can modify this path such that every critical value along the path lies on a single wall.

Note that the walls defined by the different canonical test chains are pairwise different, because the canonical test chains $\cF_\bullet$ are determined by the sets of indices for which $\cF_i$ is equal to $\cE_i$, i.e. the set of indices for which $\cE_i/\cF_i=0$ and the set of indices for which $\cF_i=0$. If we denote by $\un{m}:=\rk(\cE_\bullet/\cF_\bullet)$ and $m:= \sum m_i$ the wall is perpendicular to the vector $\frac{1}{n} \un{n} - \frac{1}{m} \un{m}$. Two such vectors can only be collinear if the sets of indices for which $m_i\neq 0$ are complementary, which does not happen for the canonical test chains.

Let $\alpha_0$ be a critical value lying on a single wall, such that  $\Chain_{\un{n}}^{\un{d},\alpha_0^+-\sst}$ contains a maximal Harder-Narasimhan stratum. Let $\cE_\bullet$ be a chain in this stratum and denote by $\cF_\bullet^j$ its Harder-Narasimhan filtration.
Then for all subquotients $\cF_\bullet^j/\cF_\bullet^k$ occurring in this stratum we know that the set $I^j=\{ i | \rk(\cF_i^j(\cF_i^k)\neq 0\}$ is a string of consecutive integers, because otherwise the subquotient would be a direct sum of chains defining different walls, which cannot happen because we assumed that $\alpha_0$ lies on a single wall.

Therefore we can apply Corollary \ref{maximalSummands} to the maximal pair $(\cE_\bullet/\cF_\bullet^1,\cF^1_\bullet)$ and conclude that for $\Chain_{\un{n}}^{\un{d},\alpha_0^+-\sst}$ one of the canonical test chains is destabilizing, so that by \cite[Proposition 4]{GPH}  $\Chain_{\un{n}}^{\un{d},\alpha_0^+-\sst}$ is empty. 

Finally, we claim that only a single maximal Harder-Narasimhan stratum occurs in $\Chain_{\un{n}}^{\un{d},\alpha_0^+-\sst}$ and that this is a Harder-Narasimhan stratum of a type given by a filtration of length 1 defined by one of the canonical test chains. To see this we note that by Corollary \ref{maximalSummands} $\cF_\bullet^1$ or $\cE_\bullet/\cF_\bullet^1$ contains a direct summand isomorphic to a canonical test chain. 

If $\cF^1_\bullet$ contains the summand $\cG_\bullet$ then $\cQ_\bullet=\cF_\bullet^1/\cG_\bullet$ is a chain of shorter length. Denote by $\un{n}_{\cG}:=\rk(\cG_\bullet)$  and $\un{n}_{\cQ}:=\rk(\cF_\bullet^1/\cG_\bullet)$ and by $n_{\cG}$, resp. $n_{\cQ}$ the rank of $\oplus \cG_i$ and $\oplus \cQ_i$. The wall defined by $\mu_\alpha(\cE_\bullet)=\mu_\alpha(\cG)$ is by definition perpendicular to $\frac{1}{n} \un{n} - \frac{1}{n_{\cG}} \un{n}_{\cG}$ and the wall defined by $\cQ_\bullet$ is perpendicular to $\frac{1}{n} \un{n} - \frac{1}{n_{\cQ}} \un{n}_{\cQ}$. By assumption these walls have to coincide because $\alpha_0$ lies on a single wall. But since there exists $i$ such that $n_{i,\cQ} =0$ and $n_{i,\cG}=n_i$ this can only happen if the  two vectors differ by a negative scalar. 
However this would imply that $\cQ_\bullet$ and $\cG_\bullet$ are destabilizing on different sides of the wall, contradicting our assumption that $\cF^1_\bullet$ was the first step of the Harder-Narasimhan flag. Thus $\cG_\bullet=\cF^1_\bullet$. Moreover, in this case the quotient $\cE_\bullet/\cF_\bullet^1$ is a shorter chain. If the Harder-Narasimhan filtration of such a quotient was non-trivial, we could repeat the argument and find a subquotient of $\cE_\bullet$ of even shorter length, still defining the same wall, which cannot happen. 

The same argument applied to $\cE_\bullet/\cF_\bullet^1$ implies that a canonical direct summand of $\cE_\bullet/\cF_\bullet^1$ must be the maximally destabilizing quotient of $\cE_\bullet$, since again if $\cE_\bullet/\cF_\bullet^1=\cQ_\bullet \oplus \cG_\bullet$ we find that the two summands must be destabilizing on different sides of the wall. However, the canonical summand can only be destabilizing for $\alpha_0^+$  if we chose $\gamma$ to end at the first intersection with a wall defined by one of the canonical test chains. Thus the summand must be the maximally destabilizing quotient of $\cE_\bullet$ and then we can conclude as above that it must define the only step in the Harder-Narasimhan filtration.

If $\un{n}=(n,\dots,n)$ is constant the argument is simpler: In this case the only canonical test chains are given by subchains of rank $(n,\dots,n,0,\dots,0)$ and by \cite[Lemma 9]{GPH} for the line $\alpha_t := \alpha + t(0,1,\dots,r)$ we know that $\Chain_{\un{n}}^{\un{d},\alpha_t-\sst}$ is irreducible for $t\gg 0$. Moreover if a maximal Harder-Narasimhan stratum with subchains of rank $(n,\dots,n,0,\dots,0)$ would occur, then all chains in $\Chain_{\un{n}}^{\un{d}}$ would be $\alpha$-unstable.
\end{proof}

\begin{proof}[Proof of Theorem 2]
This now follows by induction: A Harder-Narasimhan stratum is irreducible if and only if the stacks of stable chains parametrizing the associated graded chains are irreducible. By induction on the rank of the chains this holds true.  Thus choosing a path $\gamma$ as in Proposition \ref{path} we find that  the stacks $\Chain_{\un{n}}^{\un{d},\alpha_\infty-\sst}$ contain a unique irreducible component of dimension $\dim(\Chain_{\un{n}}^{\un{d},\alpha-\sst})$. Since no other maximal Harder-Narasimhan strata occur in the wall crossing formula for critical points in the path $\gamma$, the smooth, equidimensional stack $\Chain_{\un{n}}^{\un{d},\alpha-\sst}$ also contains a unique component of maximal dimension, so it must be irreducible.
\end{proof}

\end{document}